\documentclass[11pt,reqno]{amsart}
   \usepackage{graphicx}
\usepackage{amscd,amsmath,amsopn,amssymb,amsthm,multicol}
\usepackage{hyperref}
\usepackage[all]{xy}
\usepackage{amscd}
\usepackage{verbatim}
\usepackage{lscape}
\usepackage{slashed}
\usepackage{graphicx}
\usepackage{color}
\usepackage{lscape}
 \usepackage{cancel}
 \usepackage{musixtex}
   \usepackage{harmony}
   \usepackage{comment}
   \usepackage{musixtex}
   \usepackage{wasysym}
   \usepackage[b]{esvect}
   \usepackage{tikz-cd}
\usepackage{mathrsfs}
\DeclareMathAlphabet{\mathscrbf}{OMS}{mdugm}{b}{n}
\numberwithin{equation}{section}
\definecolor{dblue}{rgb}{0.01,0.01,0.44}
\definecolor{red}{rgb}{0.57,0.11,0.15}

\DeclareMathOperator{\diag}{diag}
\DeclareMathOperator{\Ad}{Ad}

\DeclareMathOperator{\Id}{Id}

\DeclareMathOperator{\Sym}{Sym}

\DeclareMathOperator{\Ric}{Ric}

\DeclareMathOperator{\vol}{vol}

\DeclareMathOperator{\Spin}{Spin}
\DeclareMathOperator{\rnk}{rnk}

\DeclareMathOperator{\dd}{d}

\newcommand{\fr}{\mathfrak}
\newcommand{\al}{\alpha}
\newcommand{\be}{\beta}

\newcommand{\wi}{\widetilde}
\newcommand{\bb}{\mathbb}
\newcommand{\mc}{\mathcal}

\DeclareMathOperator{\SO}{SO}
\DeclareMathOperator{\Sp}{Sp}
 \DeclareMathOperator{\SU}{SU}
  
\DeclareMathOperator{\U}{U}
\DeclareMathOperator{\G}{G}

\DeclareMathOperator{\E}{E}
\DeclareMathOperator{\Ss}{S}

\DeclareMathOperator{\Gl}{GL}
\newcommand{\thickline}{\noalign{\hrule height 1pt}}

   \newtheorem{lemma} {Lemma} [section]
\newtheorem{theorem}[lemma]{Theorem}
\newtheorem{remark}[lemma] {Remark}
\newtheorem{prop} [lemma]{Proposition}
\newtheorem{definition}[lemma] {Definition}
\newtheorem{corol}[lemma] {Corollary}
\newtheorem{example}[lemma] {Example}

 \definecolor{mygf}{rgb}{0.23, 0.31,0.47}
\definecolor{myb}{rgb}{0.09, 0.24,0.47}
\definecolor{mygrey1}{rgb}{0.13, 0.11,0.17}
\definecolor{myopen}{rgb}{0.65,0.36,0.16}
\definecolor{mydark}{rgb}{0.08,0.18,0.48}
\definecolor{drey}{rgb}{0.121,0.21,0.452}
\definecolor{mred}{rgb}{0.4,0.17,0.1}
\definecolor{ur}{rgb}{0.57,0.31,0.25}
\definecolor{red}{rgb}{0.57,0.11,0.15}
\definecolor{re}{rgb}{0.57,0.21,0.15}
\definecolor{blue}{rgb}{0.07,0.24,0.38}
\definecolor{dblue}{rgb}{0.07,0.24,0.65}
\definecolor{u}{rgb}{0.6,0.04,0.2}
\definecolor{dgreen}{rgb}{0.11,0.16,0.41}
\definecolor{dgr2}{rgb}{0.6,0.36,0.2}
\definecolor{grb}{rgb}{0.18,0.30,0.58}
\definecolor{mmg}{rgb}{0.08,0.48,0.28}
\definecolor{cr}{rgb}{0.8,0.1,0.0}
\definecolor{crew}{rgb}{0.2,0.5,0.2}
\begin{document}

\title{Decomposable $(4,7)$ solutions  in eleven-dimensional supergravity} 
 \author{Dmitri Alekseevsky, Ioannis Chrysikos, Arman Taghavi-Chabert}
  \address{Institute for Information Transmission Problems, B. Karetny
per. 19, 127051, Moscow, Russia and Faculty of Science,  University of Hradec Kr\'alov\'e,  Rokitanskeho 62, Hradec Kr\'alov\'e
 50003, Czech Republic}
 \email{dalekseevsky@iitp.ru}
  \address{Dipartimento di Matematica ``G. Peano'', Universit\`a  degli Studi di Torino, Via Carlo Alberto 10, 10123 Torino, Italy}
  \email{ioannis.chrysikos@unito.it}
  \address{Dipartimento di Matematica ``G. Peano'', Universit\`a  degli Studi di Torino, Via Carlo Alberto 10, 10123 Torino, Italy and  Department of Mathematics,
Faculty of Arts and Sciences, American University of Beirut, P.O. Box 11-0236 Riad El Solh, Beirut 1107 2020, Lebanon}
\email{taghabert@gmail.com}

 \begin{abstract}
  Consider an oriented four-dimensional   Lorentzian manifold $(\wi{M}^{3, 1}, \wi{g})$ and an oriented seven-dimensional  Riemannian manifold $(M^{7}, g)$.  We   describe a  class  of decomposable eleven-dimensional supergravity backgrounds  on  the product manifold  $(\mc{M}^{10, 1}=\wi{M}^{3,1} \times M^7, g_{\mc{M}}=\wi{g}+g)$,   endowed with  a flux  form   given in terms of the volume form on $\wi{M}^{3, 1}$ and  a     closed $4$-form $F^{4}$ on $M^{7}$.  
 We show that the Maxwell equation for such a flux form can be read in terms of the co-closed 3-form $\phi=\star_{7}F^{4}$. Moreover,  the   supergravity   equation  reduces  to the  condition  that  $(\wi{M}^{3,1},\wi{g})$ is   an Einstein  manifold  with negative Einstein  constant
   and $(M^7, g, F)$
    is  a Riemannian manifold  which satisfies   the  Einstein   equation    with   a  stress-energy tensor   associated  to the 3-form $\phi$.  Whenever this 3-form is generic,  the Maxwell  equation    induces   a  weak   $\G_2$-structure   on $M^{7}$ and  then we obtain decomposable  supergravity backgrounds given by the product of  a   weak $\G_2$-manifold $(M^7, \phi, g)$   with  a Lorentzian   Einstein manifold  $(\wi{M}^{3,1},\wi{g})$.  
    We classify  homogeneous 7-manifolds  $M^{7}=G/H$ of a compact Lie group $G$ and indicate the cosets which admit  an invariant or non-invariant $\G_2$-structure, or even no $\G_2$-structure.    Then we construct  examples  of    compact  homogeneous  Riemannian  7-manifolds endowed with  non-generic invariant  3-forms which satisfy  the Maxwell  equation, but  the construction of   decomposable     homogeneous supergravity backgrounds of this type remains an open problem.


    \end{abstract}
\maketitle

\section{Introduction}\label{intro}

  \vskip 0.1cm
 Ten-dimensional supersymmetric string theories and their eleven-dimensional unified analogue, called M-theory,  are some of the most  promising approaches to  a consistent model for the unification of fundamental forces of nature.
Indeed, supergravity theories merge the theory of general relativity with supersymmetry and are crucial for understanding the dynamics of massless fields in string theories, since they  determine the appropriate backgrounds in which strings propagate (see  \cite{BBS} for a comprehensive survey).   Nowadays there are several known consistent supergravity theories in different dimensions. For example,  in dimension ten there are at least 5 different types of   string theories, namely Type I, Type IIA and IIB and some heterotic $\E_8 \times \E_8$ and $\SO_{32}$ theories.  In dimension eleven  physicists  are concerned with  the (weak) coupling limits  of these theories via   T-duality and other kinds of dualities that  yield a unique eleven-dimensional M-theory.

The eleven-dimensional supergravity theory has as  bosonic fields some   Lorentzian metric  $g_{\mc{M}}$ and a 3-form potential $A$ with 4-form field strength $\mc{F}=\dd A$, the so-called {\it flux form}, 
satisfying  the supergravity field equations (with zero gravitino):
\[
\left\{
\begin{tabular} {rclll}
$\dd \mc{F}$ & $=$ &  $0,$ & $ \text{{\it Closure}}$ &  $(\mathscr{C})$, \\
$\dd \star \mc{F}$ &   $=$ & $(1/2)\mc{F} \wedge \mc{F},$ & $ \text{{\it Maxwell}}$ & $(\mathscr{M})$, \\
$\Ric^{g_{\mc{M}}}(X,Y)$ &   $=$ & $(1/2) \langle X\lrcorner \mc{F}, Y\lrcorner \mc{F} \rangle -(1/6)g_{\mc{M}}(X,Y) \| \mc{F} \|^2$, & $\text{{\it Einstein}}$  & $(\mathscr{E})$.
\end{tabular}\right.
\]
Here, $\dd\equiv \dd^{g_{\mc{M}}}$ is the exterior derivative of differential forms on the Lorentzian manifold $(\mc{M}^{10, 1}, g_{\mc{M}})$,  $\Ric^{g_{\mc{M}}}$ is the Ricci tensor of the Levi-Civita connection on $\mc{M}$, and
\[
\langle X\lrcorner \mc{F}, Y\lrcorner \mc{F}  \rangle   = \frac{1}{3!} g_{\mc{M}}(X\lrcorner \mc{F}, Y\lrcorner \mc{F}), \quad\quad  \|\mc{F} \|^2   = \frac{1}{4!} g_{\mc{M}}(\mc{F}, \mc{F}).
\]

\noindent The second equation is referred to as  the {\it Maxwell-like equation} and  the third one as the {\it supergravity Einstein equation}.   Note that usually one asks from $\mc{M}^{10, 1}$ to be also  spin, but in this work we are not interested in the supersymmetries of the model, so we do not pay much attention to this condition.

Classification of supergravity backgrounds, i.e. Lorentzian manifolds $(\mc{M}^{10, 1}, g_{\mc{M}}, \mc{F}^{4})$ solving the above system, can be considered in several different contexts. For example, besides the construction of Killing superalgebras  (see \cite{FOF, FOFP}), there are also methods  based on the theory of $G$-structures (see for example \cite{Duff2, BJ, Pap, Mac, Wit}).  In  this paper   we are concerned with      eleven-dimensional oriented Lorentzian manifolds $\mc{M}\equiv\mc{M}^{10, 1}:=\wi{M}^{3,1} \times M^7$ given by a product of a four-dimensional oriented Lorentzian  manifold $(\wi{M}\equiv\wi{M}^{3, 1}, \wi{g})$ and a seven-dimensional (compact)  oriented Riemannian manifold $(M\equiv M^{7}, g)$ and  analyse the supergravity equations from a purely geometric perspective.
 In particular, we consider the following  type of   flux  forms on $\mc{M}$
  \[
  \mc{F}^{4}=f\cdot\vol_{\wi{M}}+F^{4}, \quad\quad (\ast)
  \]
where $F^{4}$ is a closed $4$-form on $M$ and $f\in\bb{R}$ is assumed to be a constant. 
Solutions  of eleven-dimensional supergravity for such 4-forms and  with respect to the product metric $g_{\mc{M}}=\wi{g}+g$, will be called {\it $(4,7)$-decomposable supergravity backgrounds}.

For this specific Ansatz  the core observation (see Proposition \ref{nice1}) is that the Maxwell equation $(\mathscr{M})$ is equivalent to the equation
\[
\dd\star_{7}{F}^{4}=f\cdot {F}^{4},
\]
which by setting $\phi:=\star_{7}F^{4}$ can be rewritten as
\[
\dd {\phi} = f \star_7 {\phi}. \quad\quad (\ast\ast)
\]
Moreover, the  closure condition $(\mathscr{C})$ of $\mc{F}$   can be rephrased as  $\dd \star_7 {\phi} =0$.  For brevity, 3-forms on $M^{7}$ satisfying the last two conditions for some constant $f\in\bb{R}$, will be referred to as {\it special 3-forms}. In these terms  one has that  the specific flux form $\mc{F}$  is a solution of the closure condition $(\mathscr{C})$ and the supergravity Maxwell equation $(\mathscr{M})$  if and only if    the associated 3-form $\phi:=\star_{7}F^{4}$ on $M^{7}$  is special.\\

Turning  now to  the corresponding supergravity Einstein equation $(\mathscr{E})$,  we  conclude   that   the four-dimensional Lorentzian manifold $(\wi{M}, \wi{g})$ must be  Einstein with negative Einstein constant $\Lambda:= - \frac{1}{6} \left(2f^2
+ \|{\phi}\|^2 \right)$ (Proposition \ref{true1?}).   Moreover, we see that the Ricci tensor of $(M, g)$ must satisfy the equation
\[
 \Ric^{g}(X, Y)= \frac{1}{6}g(X,Y) \left(f^2  + 2 \| {\phi} \|^2_{M} \right)+q_{\phi}(X, Y), \quad\quad (\ast\ast\ast)
\]
 where $q_{\phi}(X, Y)$ is the symmetric bilinear form defined by $q_{\phi}(X, Y):=- \frac{1}{2} \langle  X\lrcorner {\phi} , Y\lrcorner  {\phi} \rangle_{M}$.
We then  proceed with a description of some special situations arising by focussing  on $(\ast\ast\ast)$.  In particular, we  examine  the following  basic  classes  of  special 3-forms on $(M,g)$:
\begin{itemize}
\item the trivial 3-form, i.e.\ $\phi =0$ (and hence $F=0$) but with $f\neq 0$,
\item non-zero  harmonic 3-forms, i.e.\  $\phi \neq 0$, $f=0$,
\item non-harmonic  3-forms, i.e.\ $\phi \neq 0$, $f\neq 0$.
\end{itemize}
For  these three cases  we analyse  the supergravity equations and  describe solutions.
In particular, for the  more general third case  the construction of $(4, 7)$-decomposable supergravity backgrounds relies on  the theory of  $\G_2$-structures (see also \cite{Atiyah, Sfetsos, BJ,  Agr, House, Duff} for the role of $\G_2$-geometries in M-theory).
 Here,  we show that whenever  $\phi:=\star_{7}F^{4}$ is   a  co-closed {\it generic} 3-form on $M^{7}$ satisfying equation $(\ast\ast)$ for $f\neq 0$, i.e. a {\it generic special 3-form with $f\neq 0$}, which is equivalent to say that $\phi$ induces a weak  $\G_2$-structure on $M$,  then the pair
\[
(\mc{M}=\wi{M}\times M, \  g_{\mc{M}}=\wi{g}+g),
\]
where $g$ is the Einstein metric induced by $\phi$, provides $(4, 7)$-decomposable supergravity solutions. 
In particular,   we obtain that

\vskip 0.2cm
\noindent{\sc  Theorem A.}  {\it Assume that  the product  $(\mc{M}=\wi{M}\times M, g_{\mc{M}}=\wi{g}+g)$ is endowed with the  4-form  $\mc{F}^{4}:=f\cdot\vol_{\wi{M}}+F^{4}$,
for some constant $0\neq f\in\bb{R}$ and some closed 4-form $F^{4}\in\Omega^{4}_{\rm cl}(M)$ on $M$, such that $\phi:=\star_{7}F^{4}$ is a generic 3-form on $M$. Then $(\mc{M}, g_{\mc{M}}, \mc{F}^{4})$   gives rise to a $(4, 7)$-decomposable supergravity background if and only if $(M, g, \phi:=\star_{7}F^{4})$ is a weak $\G_2$-manifold and $(\wi{M}, \wi{g})$ is Einstein with negative Einstein constant. In particular,  $f$ takes the values $f=\pm 2$. }

\vskip 0.2cm
 Weak $\G_2$-structures are spin 7-manifolds $(M, g, \phi)$ endowed with a generic 3-form $\phi$ satisfying the differential equation  $\dd\phi=\lambda\star_{7}\phi$, for some non-zero constant $\lambda$. Such $\G_2$-structures  are extremely interesting in theoretical and mathematical physics, since  they are manifolds admitting  non-trivial solutions of the Killing spinor equation (see \cite{FKMS}). We should emphasize that our approach to Theorem A does {\it not} take into account the theory of Killing superalgebras, i.e. we reach Theorem A by solving only  the zero gravitino supergravity equations, independently of the supersymmetries that preserves the corresponding model $\mc{M}$.  Moreover, our Ansatz serves well the purpose of finding obstructions to the existence of $(4, 7)$-decomposable supergravity backgrounds. For example,  whenever $\phi=\star_{7}F^{4}$ is a {\it generic special 3-form with $f=0$}, which means that it induces a {\it parallel $\G_2$-structure} on $M$,  we obtain the following   non-existence result.

 \vskip 0.2cm
\noindent{\sc {Corollary A.}} {\it If $f=0$ and $\phi:=\star_{7}F^{4}$ is a generic 3-form on $M^{7}$, where $F^{4}\in\Omega^{4}_{\rm cl}(M^{7})$, then the closure condition $(\mathscr{C})$ and the Maxwell equation $(\mathscr{M})$ for our Ansatz $(\ast)$, imply  that $\phi$ is $\nabla^{g}$-parallel, i.e. $\phi$ induces  a parallel $\G_2$-structures   and hence $(M, g)$ is Ricci flat.  In this  case the eleven-dimensional Lorentzian manifold $(\mc{M}=\wi{M} \times M, g_{\mc{M}}=\wi{g}+g, \mc{F}^{4})$ does not give rise to a $(4, 7)$-decomposable supergravity background.}

\vskip 0.2cm
The rest of the article is devoted to the homogeneous case, where  the calculations  related to the supergravity equations become  more attractive, since the tensor fields $g_{\mc{M}}$ and $\mc{F}^{4}$ are invariant under the action of a Lie group.
In this case we obtain a series of examples serving Theorem A, and these are based on the  the classification of compact homogeneous weak $\G_2$-manifolds and homogeneous Lorentz Einstein 4-manifolds, given in  \cite{FKMS} and \cite{Kom, Fels}, respectively.    Then we  examine     the supergravity equations for invariant {\it non-generic 3-forms} $\phi:=\star_{7}F^{4}$. 
To this end,  we classify all almost effective seven-dimensional  homogeneous   manifolds $M^{7}=G/H$  of  a compact  Lie group  $G$  (see Table 2 and Theorem \ref{classg2}). This  extends the      classification   of simply-connected homogeneous 7-manifolds  $M^{7}=G/H$ of a  semisimple compact  group $G$, which was used  for classifying   homogeneous  Einstein 7-manifolds,  see \cite{Cast2, Nik}.
In combination with the  classification of  compact homogeneous 7-manifolds  admitting invariant $\G_2$-structures given in  \cite{Le, Rei}, we  obtain the complete list of all compact (almost) effective homogeneous 7-manifolds   which admit a $\G_2$-structure but  no  invariant $\G_2$-structure (and hence no invariant  spin   structure, see Theorem \ref{NONGEN}). 
 We then   describe    all  invariant     special  3-forms $\phi$  (i.e. solutions  of  Maxwell   equation)  on   the non-spin  manifold
   $\bb{C}P^{2}\times\Ss^{3} =  \SU_3/\U_2 \times \SU_2$. We also discuss the case of the Lie group $\Ss^{3}\times{\rm T}^{4}=\SU_2\times{\rm T}^{4}$. 
  In both cases we show  that  there  are   invariant  special 3-forms   which are not generic.


    \smallskip

 \noindent {\bf Acknowledgements:} It is pleasure to thank Anna Fino (Torino) for useful references and discussions. 
 The  first author  is  partially supported  by grant  no.~18-00496S of the Czech Science Foundation.
 The second author is  a Marie Curie fellow of the Istituto Nazionale di Alta Matematica ({INdAM}) and thanks Dipartimento di Matematica  ``G. Peano''  (Universit\`a degli Studi di Torino) for its hospitality.

\section{11D supergravity backgrounds of the form $\mc{M}^{10, 1}=\widetilde{M}^{3, 1}\times{M}^{7}$}



We begin by fixing some conventions, relevant to our subsequent computations.

\smallskip
\noindent {\bf Conventions.}
 Consider an $n$-dimensional pseudo-Riemannian manifold $(N, h)$   of signature $(p,q)$. At any point $x\in N$,  the tangent space $V:=T_{x}N= \mathbb{R}^{p,q}$ $(n=p+q)$  is a pseudo-Euclidean vector  space  endowed with a non-degenerate inner product of signature
\[
(p,q)=(n-q, q)= (+\cdots +, -\cdots-).
\]
When the signature is $(n,0)$ (resp. $(n-1, 1)$), then we say that $(N, h)$ is a \emph{Riemannian} (resp. \emph{Lorentzian}) manifold.  We shall denote by $\fr{so}(V)$  the Lie algebra of skew-symmetric endomorphisms of $V$; for any $u, v\in V$ let $w\wedge u$ the skew-symmetric endomorphism on $V$, given by  $(u\wedge v)(z)=h(v, z)u-h(u, z)v$.  Hence, here we take the convention  $\omega_1 \wedge \omega_2 := \omega_1 \otimes \omega_2 - \omega_1 \otimes \omega_2$ for any two elements $\omega_{1}, \omega_2\in\bigwedge^{1}T_{x}^{*}N$. The metric tensor $h$ induces a metric in $\bigwedge^{\bullet}TN$ and its dual,  namely
\[
\langle \phi , \psi  \rangle:= \det ( \langle \phi_i , \psi_j \rangle ) = \frac{1}{k! } h(\phi , \psi),
\]
 for any  decomposable $k$-vector $\phi=\phi_1 \wedge \ldots\wedge \phi_k$ and $\psi = \psi_1 \wedge \ldots \wedge\psi_k$.
We choose a volume form $\vol^{(n)}$ normalised as $\langle \vol^{(n)} , \vol^{(n)} \rangle  = (-1)^q$. Equivalently, if $\{ e_1 , \ldots , e_p , e_{p+1} , \ldots e_{p+q} \}$ is a pseudo-orthonormal frame with
\[
h( e_i , e_j )   = \delta_{ij}, \quad h( e_k , e_\ell )   = - \delta_{k\ell}, \quad    h( e_i , e_k )   = 0, \quad  \text{for} \ \  1 \leq i,j \leq p, \  p+1 \leq k,\ell \leq p+q,
\]
then $\vol^{(n)} ( e_1 , e_2 , \ldots , e_n )  = 1$.
The Hodge star operator is defined by $\phi \wedge \star \psi  =  \left\langle \phi , \psi\right\rangle \vol ^{(n)}$ for any $k$-form $\phi$ and $\psi$. In particular,  for any $\phi\in\bigwedge^{k}T_{x}^{*}N$ we have the identities
\[
\star 1 = \vol^{(n)}, \quad \star \vol^{(n)} = (-1)^q, \quad \star \star \phi  = (-1)^{k(n-k)+q}  \phi,
\]
  and hence  $\phi \wedge \psi  = (-1)^{k(n-k)+q} \langle \phi , \star \psi \rangle \vol^{(n)}$,
  for any $\phi\in\bigwedge^{k}T_{x}^{*}N$ and $\psi\in\bigwedge^{n-k}T_{x}^{*}N$.  

\subsection{Supergravity backgrounds   of the form $\mc{M}^{10, 1}=\widetilde{M}^{3, 1}\times{M}^{7}$ }   

Let us consider an  eleven-dimensional Lorentzian manifold  $(\mc{M}\equiv\mc{M}^{10, 1}, g_{\mc{M}})$ given by the  product of a four-dimensional Lorentzian manifold $(\wi{M}\equiv\wi{M}^{3, 1}, \wi{g})$ and a seven-dimensional Riemannian manifold $(M\equiv M^{7},  g)$,
\begin{equation}\label{product}
(\mc{M}, g_{\mc{M}}) = (\wi{M} \times M, \ g_{\mc{M}}:=\wi{g}+ g).
\end{equation}
 We assume that both $(\wi{M}, \wi{g})$ and $(M, g)$ are oriented with volume forms $\vol_{\wi{M}}$ and $\vol_{M}$, respectively.  Then,  the volume form on $\mc{M}$ is given by $\vol_{\mc{M}} := \vol_{\wi{M}}+\vol_{M}$ and $\mc{M}$ is oriented as well. Since  $\dim\wi{M}=4$,  notice that any 4-form on $\wi{M}^{4}$ is closed. We mention that we  do {\it not} assume any {\it homogeneity condition} for the Lorentzian manifold $\mc{M}=\wi{M}^{}\times M^{}$.
However, we   will   assume  that   $M^7$  is   compact   and  that    the   flux  4-form is  given  by
\begin{equation}\label{maincase1}
\mc{F}^{4}:=f\cdot \vol_{\wi{M}}+ {F}^{4},
\end{equation}
for some closed 4-form $F^{4}$ on $M$ and  a  constant  $f\in\bb{R}$. Note that the last condition  is equivalent to say that $\wi{F}^{4}$ is co-closed, i.e. $\dd\star_{4}\wi{F}^{4}=0$, where $\star_{4} : \Omega^{k}(\wi{M})\to \Omega^{4-k}(\wi{M})$ is the Hodge star operator on $\wi{M}$.  Indeed, $\star_{4}^{2}\big|_{\Omega^{k}}=(-1)^{k(4-k)+1}\Id_{\Omega^{k}}$,
with $\star_{4}\vol_{\wi{M}^{4}}=(-1)^{q}=-1$ (since  $q=1$), and hence the relation   $\wi{F}^{4}:=f\cdot\vol_{\wi{M}}$ yields  $\star_{4} \wi{F}^{4}=-f$. Next we shall call 4-forms of  type (\ref{maincase1}) {\it decomposable}.


 \vskip 0.2cm
 \noindent{\bf On the closure condition $(\mathscr{C})$ and the Maxwell equation $(\mathscr{M})$.}
Let us focus now on the closure condition $(\mathscr{C})$ and the Maxwell equation  $(\mathscr{M})$. We  denote   the Hodge star operators    on $\mc{M}$ and ${M}$   as    $\star_{11} : \Omega^{k}(\mc{M})\to \Omega^{11-k}(\mc{M})$ and $\star_{7} : \Omega^{k}({M})\to \Omega^{7-k}({M})$, respectively.
 We  need  the following elementary  result (which  makes sense, appropriately reformulated,   for   any  pseudo-Riemannian metric).


\begin{lemma} \label{work1}
Consider the Lorentzian manifold $(\mc{M}^{10, 1}=\wi{M}^{3, 1}\times M^{7}, g_{\mc{M}}=\wi{g}+g)$ and let $\wi{\alpha} \in \Omega^k (\wi{M})$ and $\alpha \in \Omega^\ell (M)$ be some differential forms of $\wi{M}$ and $M$, respectively.  Then, since $T\mc{M}=T\wi{M}\oplus TM$ defines a decomposition of the tangent bundle of $\mc{M}$, the following holds:\\
(1) \[
g_{\mc{M}} ( \wi{\alpha} \wedge {\alpha} , \wi{\alpha} \wedge {\alpha}) = \frac{(k+\ell)!}{k!\ell!} \ \wi{g} ( \wi{\alpha} , \wi{\alpha} )\cdot  g ({\alpha}, {\alpha}) \,.
\]
and consequently,
\[
\langle \wi{\alpha} \wedge {\alpha} , \wi{\alpha} \wedge {\alpha} \rangle_{\mc{M}} = \langle \wi{\alpha},   \wi{\alpha} \rangle_{\wi{M}} \cdot \langle {\alpha}, {\alpha} \rangle_{M}, \quad
 \|\wi{\al}^{k}\wedge\al^{\ell}\|_{\mc{M}}=\|\wi{\al}^{k}\|_{\wi{M}}\cdot\|\al^{\ell}\|_{M}.
 \]
(2) The action of the Hodge star operator  $\star_{11} : \Omega^{r}(\mc{M})\to \Omega^{11-r}(\mc{M})$ on $\wi{\al}^{k}\wedge\al^{\ell}$ reads as
\[
\star_{11}(\wi{\al}\wedge\al)=(-1)^{\ell(p-k)}\star_{p}\wi{\al}\wedge\star_{11-p}\al.
\]
\end{lemma}
\noindent Now we are ready to prove that

\begin{prop}\label{nice1}
For the     4-form on $\mc{M}=\wi{M}+{M}$ given by the Ansatz (\ref{maincase1}) with $f\in\bb{R}$, the  closure condition $(\mathscr{C})$ and the Maxwell equation  $(\mathscr{M})$ are simultaneously satisfied, if and only if
\begin{equation}\label{weakg2}
\dd F^{4}=0, \quad \text{and}\quad \dd\star_{7}{F}^{4}=f\cdot {F}^{4}.
\end{equation}
 In the case where $f=0$, then the equations $(\mathscr{C})$ and $(\mathscr{M})$  are simultaneously  satisfied if and only if the 4-form $F^{4}$ on $M^{7}$ is closed and co-closed, $dF^{4}=\dd{\star_{{7}}}F^{4}=0$.
 \end{prop}
\begin{proof} Let us compute  $\star_{11}\mc{F}$. We write $\mc{F}^{4}=f\cdot\vol_{\wi{M}}\wedge 1+ F^{4}\wedge \wi{1}$. Thus, by Lemma \ref{work1} and since $\omega_1\wedge\omega_2=(-1)^{st}\omega_2\wedge\omega_1$ for some $s$-form $\omega_1$ and $t$-form $\omega_2$, we conclude that
\begin{eqnarray*}
\star_{11}\mc{F}^{4}&=&\star_{4}(f\cdot\vol_{\wi{M}}\wedge 1)+\star_{7}(F^{4}\wedge \wi{1})=f\cdot\star_{4}(\vol_{\wi{M}}\wedge 1)+\star_{7}(\wi{1}\wedge F^{4})\\
&=&f\cdot[(-1)^{0(4-4)}\star_{4}\vol_{\wi{M}}\wedge \star_{7}1]+[(-1)^{4(4-0)}\star_{4}\wi{1}\wedge \star_{7}F^{4}]\\
&=&f\cdot[-\wi{1}\wedge\vol_{M}]+\vol_{\wi{M}}\wedge \star_{7}F^{4}\\
&=&-f\cdot\vol_{M}+\vol_{\wi{M}}\wedge \star_{7}F^{4},
\end{eqnarray*}
where we used the identity $\star_{4}\wi{1}=\vol_{\wi{M}}$. Thus $\star_{11}\mc{F}=-f\cdot\vol_{M}+\vol_{\wi{M}}\wedge \star_{7}{F}^{4}$ and consequently $ \dd\star_{11}\mc{F}=\vol_{\wi{M}}\wedge \dd\star_{7}{F}^{4}$. We also  compute $\mc{F}\wedge\mc{F}=2\cdot f\cdot \vol_{\wi{M}}\wedge  {F}^{4}$.
  Therefore, for our Ansatz (\ref{maincase1})  the Maxwell equation  $\dd\star_{11}\mc{F}=\frac{1}{2}\mc{F}\wedge\mc{F}$ is equivalent to $
\vol_{\wi{M}}\wedge \dd\star_{7}{F}^{4}=f\cdot \vol_{\wi{M}}\wedge {F}^{4}$,
 and our assertion is immediate.
\end{proof}


For the following, let us denote   the 3-form $ \star_7 {F}^{4}$  by $\phi:= \star_7 {F}^{4}$.
 Since the square of the star operator $\star_{7}$ acts by $\star_{7}^{2}\big|_{\Omega^{p}(M^{7})}=(-1)^{p(7-p)}\Id_{\Omega^{p}(M^{7})}$, we get that $\star_{7}\phi=\star_{7}^{2}F^{4}=(-1)^{4(7-4)}F^{4}=F^{4}$.  Thus, by Proposition \ref{nice1} we deduce that
\begin{corol}\label{req}
The Maxwell equation  $(\mathscr{M})$  for the 4-form $\mc{F}$ given by (\ref{maincase1}), i.e.  the second relation in (\ref{weakg2}),   is equivalent to the equation 
\begin{equation}\label{eq-dphi=sphi}
\dd {\phi} = f \star_7 {\phi},
\end{equation}
for the 3-form $\phi:= \star_7 {F}^{4}$.
Moreover, the closure condition $(\mathscr{C})$  is equivalent to the relation 
\begin{equation} \label{eq-dsphi=0}
 \dd \star_7 {\phi} =0.
\end{equation}
\end{corol}
\noindent This motivates us to introduce the following definition.

\begin{definition}
    \textnormal{A  3-form $\phi \in  \Omega^3(M)$ on  a Riemannian 7-manifold $(M,g)$  is called} {\it special}   \textnormal{if  it is  co-closed ($\dd \star_{7}\phi =0$) and satisfies the relation  $\dd \phi =  f\star_{7}\phi$ for some constant  $f\in\bb{R}$.}
     \end{definition}

\noindent In  terms of special 3-forms, Corollary \ref{req} reads as follows: 
\begin{corol}\label{special3}
 The 4-form $ \mc{F}=f\cdot\vol_{\wi{M}}+F^{4}\in\Omega^{4}_{\rm cl}(\mc{M})$ for some constant $f$ and closed 4-form $F^{4}\in \Omega^{4}_{\rm cl}(M^{7})$,
is a solution of Maxwell equation $(\mathscr{M})$ if and only if $\phi:=\star_{7}F$ is a special 3-form on $M^{7}$. %
\end{corol}
 \vskip 0.3cm

\noindent {\bf On the Einstein supergravity equation $(\mathscr{E})$.}
 For the computations related to the  right hand side of the Einstein supergravity equation  $(\mathscr{E})$   we use the following basic lemma.

\begin{lemma}\label{lem2}
Let $\phi$ be a $k$-form on a smooth pseudo-Riemannian manifold $(M^{p,q}, g)$ of signature $(p,q)$ with $p+q=n$. When $1 \leq k \leq n-1$, we have
\begin{align}\label{eq-to-show2}
(-1)^q \langle X \lrcorner \star \phi , Y \lrcorner \star \phi \rangle & = \langle \phi, \phi \rangle \langle X , Y \rangle -  \langle X \lrcorner \phi , Y \lrcorner \phi \rangle \, , & \mbox{for all vector fields $X$ and $Y$.}
\end{align}
When $k=n$, we have
\begin{align}\label{eq-to-show1}
\langle X \lrcorner \phi , Y \lrcorner \phi \rangle & = \langle \phi , \phi \rangle \langle X , Y \rangle \, , & \mbox{for all vector fields $X$ and $Y$.}
\end{align}
\end{lemma}

\begin{proof}
It suffices to prove \eqref{eq-to-show2} and \eqref{eq-to-show1} by taking $X$ and $Y$ to be basis elements at a point. Let us fix an orthonormal basis $\{ e_i \}_{i=1, \ldots , n}$ with $\langle e_i , e_j \rangle = \delta_{ij}$ for $1 \leq i,j \leq p$, and $\langle e_i , e_j \rangle = -\delta_{ij}$ for $p+1 \leq i,j \leq p+q$. Denote by $\{ e^i \}_{i=1, \ldots , n}$ the corresponding dual basis such that the volume form is given by $\vol =  e^1 \wedge \ldots \wedge e^n$. For any $1 \leq k \leq n$, the $k$-forms $\{ e^{i_1} \wedge \ldots \wedge e^{i_k} \}$ constitute a basis for $\bigwedge^k T M$ orthonormal with respect to   the natural  extension   $\langle \cdot , \cdot \rangle$   of  the  metric , i.e.\
\begin{align*}
\langle e^{i_1} \wedge \ldots \wedge e^{i_k} , e^{i_1} \wedge \ldots \wedge e^{i_k}\rangle & =
(-1)^u \, ,
\end{align*}
where $u$ is the number of  timelike $1$-forms among the $\{ e^i \}$. In the following discussion,  $\{ i_1 , \ldots , i_n \}$ will denote an even permutation of $\{ 1 \ldots n \}$. For any $1 \leq k \leq n$, set $I=\{ i_1 , \ldots , i_k\}$ and $J=\{ i_{k+1} , \ldots , i_n\}$ so that $I \cap J = \emptyset$. Then
\begin{align*}
\star ( e^{i_1} \wedge \ldots \wedge e^{i_k} ) = (-1)^u e^{i_{k+1}} \wedge \ldots \wedge e^{i_n}.
\end{align*}

\noindent Let us deal with the case $1 \leq k \leq n-1$ first. By invariance, we may assume that
\[
 \langle X \lrcorner \star \phi , Y \lrcorner \star \phi \rangle  =a \langle \phi, \phi \rangle \langle X , Y \rangle +b \langle X \lrcorner \phi , Y \lrcorner \phi \rangle \,
 \]
for some $a, b\in\bb{R}$. To determine $a$ and $b$ we choose $\phi$ be a basis element, i.e.\ $\phi = e^{i_1} \wedge \ldots \wedge e^{i_k}$. It is also clear that if $X$ and $Y$ are linearly independent, then each term of this expression vanishes. Hence, we may take $X = Y = e_r$  for any $1 \leq r \leq n$.
Then, it is easy to check the following:
\begin{itemize}
\item If $r \in I$, then we have $0 = (-1)^u a + (-1)^{u} b$ when $e_r$ is spacelike, and $0 = a (-1)^u (-1)  +  b (-1)^{u-1}$ when $e_r$ is timelike,  so we must deduce $a=-b$ in both cases.
\item If $r \in J$, then we have $(-1)^{q-u} = a (-1)^u + 0$ when $e_r $ is spacelike, and $(-1)^{q-u-1} = a (-1)^u (-1) + 0$ when $e_r $ is timelike. Hence, in both cases, $a = (-1)^q$.
\end{itemize}
Therefore, $a= (-1)^q$ and $b = -(-1)^q$, which proves the claim.  We leave it to the reader to check \eqref{eq-to-show1}, which is completely analogous (here, one takes $\phi$ to be $e^1 \wedge \ldots \wedge e^n$).
\end{proof}

\noindent Applying Lemma \ref{lem2}   in our case, we obtain the following useful corollary.
\begin{corol}\label{impo1}
The 4-forms $\wi{F}^{4}=f\cdot\vol_{\wi{M}}\in\Omega^{4}(\wi{M})$ and  $F^{4}=\star_{7}\phi\in\Omega^{4}_{\rm cl}(M)$ satisfy the following relations
\begin{eqnarray*}
\langle  X\lrcorner \wi{F}, Y\lrcorner\wi{F}\rangle_{\wi{M}} &=&f^{2}\|\vol_{\wi{M}}\|^{2}_{\wi{M}}\wi{g}(X, Y)=-f^{2}\wi{g}(X, Y), \quad\forall \ X, Y\in\Gamma(T\wi{M}), \\
 \langle  X\lrcorner  {F} , Y\lrcorner  {F} \rangle_{M} &=& g(X , Y)\|{\phi}\|^{2}_{M} - \langle X\lrcorner {\phi} , Y\lrcorner {\phi} \rangle_{M}, \quad\forall\ X, Y\in\Gamma(T{M}).
\end{eqnarray*}
Moreover,  $\|{F}\|_{M}^2 =\|\star_{7}\phi\|^{2}_{M}= \| {\phi} \|^2_{M}$ and
\[
 \|\mc{F}\|_{\mc{M}}^{2}=\langle \mc{F}, \mc{F}\rangle_{\mc{M}}= \langle f \cdot \vol_{\wi{M}}  + F^4, f\cdot \vol_{\wi{M}} + F^4\rangle_{\mc{M}} = -f^{2}+\|F^{4}\|^{2}_{M}.
\]
 \end{corol}

 Now, for the Lorentzian manifold   $(\mc{M}=\wi{M}\times M, g_{\mc{M}}=\wi{g}+g)$ the Levi-Civita connection $\nabla^{g_{\mc{M}}}$ splits as $\nabla^{g_{\mc{M}}} = \nabla^{\wi{g}} + \nabla^g$, where $\nabla^{\wi{g}}$ and $\nabla^{g}$ are the Levi-Civita connections on $(\wi{M}, \wi{g})$ and $(M, g)$, respectively. This effects on the Ricci tensor $\Ric^{g_{\mc{M}}}$ of $\nabla^{g_{\mc{M}}}$, which  splits accordingly, i.e.
\begin{align*}
\Ric^{g_{\mc{M}}}(X,Y) & = 0 \, , & \mbox{for any vector field $X$ on $\wi{M}$ and $Y$ on $M$,} \\
\Ric^{g_{\mc{M}}}(X,Y) & = \Ric^{\wi{g}} (X,Y) \, , & \mbox{for any vector field $X, Y$ on $\wi{M}$,} \\
\Ric^{g_{\mc{M}}}(X,Y) & = \Ric^{g}(X,Y) \, , & \mbox{for any vector field $X, Y$ on $M$.}
\end{align*}
 Initially we examine  the  Einstein supergravity equation $(\mathscr{E})$ for  some  vector fields $X,Y$ on $\wi{M}$. In this case for the  Lorentzian 4-manifold $(\wi{M}, \wi{g})$ we deduce that
\begin{prop}\label{true1?}
 Let  $(\wi{M}, \wi{g}, \wi{F}^{4}=f\cdot\vol_{\wi{M}})$ be the four-dimensional Lorentzian manifold of an eleven-dimensional  supergravity background of the form $(\mc{M}=\wi{M}\times M, g_{\mc{M}}=\wi{g}+g)$, where the flux 4-form $\mc{F}$ is given by (\ref{maincase1}), with $f\in\bb{R}$. Then, $(\wi{M}, \wi{g})$ is  Einstein with negative Einstein constant $\Lambda:= - \frac{1}{6} \left(2 f^2
+ \| {\phi} \|^2 \right)$.  In particular, $\|\phi\|$ is constant.
\end{prop}
\begin{proof}
Since we can always write $F=\star_{7}\phi$ for some (co-closed) 3-form $\phi$ on $M^{7}$, the proof is based on the previous observations.   In particular, a direct computation in combination with  Corollary \ref{impo1}, shows that
\begin{eqnarray*}
\Ric^{\wi{g}}(X,Y) &=& \frac{1}{2} \langle f\cdot X\lrcorner \vol_{\wi{M}} , f\cdot Y\lrcorner \vol_{\wi{M}}  \rangle_{\wi{M}}  - \frac{1}{6} \wi{g}(X,Y) \left( \| f\cdot\vol_{\wi{M}} \|^2_{\wi{M}}+ \| {F} \|^2_{M} \right)\\
&=&- \frac{1}{2} f^2 \wi{g} (X,Y)
+ \frac{1}{6}\wi{g}(X,Y) \left( f^2  - \|F \|^2_{M} \right) \nonumber \\
& =&  \frac{1}{6} \left( -2 f^2
- \|F \|_{M}^2 \right)  \wi{g}(X,Y) = \frac{1}{6} \left(-2f^2
- \| {\phi} \|^2 \right) \wi{g}(X, Y).
\end{eqnarray*}
The constancy of $\|\phi\|$ follows easily.
\end{proof}
\noindent Therefore, the supergravity Einstein equation $(\mathscr{E})$ for the specific flux form $\mc{F}^{4}$ given by (\ref{maincase1}), forces the  Lorentzian 4-manifold $(\wi{M}, \wi{g})$ to be  Einstein.  We mention that this occurs independently of the closure condition $(\mathscr{C})$ for $\mc{F}$, or the Maxwell equation $(\mathscr{M})$,  so it is independent of the notion of special 3-forms. However, it yields the constraint   $\|\phi\|=\text{constant}$.

Let us restrict now the supergravity Einstein equation  $(\mathscr{E})$ on vector fields $X, Y\in\Gamma(TM^{7})$. Since $F=\star_{7}\phi$, by Corollary \ref{impo1} it follows that
 \begin{eqnarray*}
 \Ric^{g} (X,Y) & = &
\frac{1}{2} \langle  X\lrcorner F , Y\lrcorner F \rangle_{M}
- \frac{1}{6}g(X,Y) \left(-f^{2}  + \|F\|_{M}^2 \right) \nonumber \\
& =& \frac{1}{2} \langle  X\lrcorner \star_{7}\phi , Y\lrcorner \star_{7}\phi \rangle_{M}
+ \frac{1}{6} g(X,Y) \left(f^{2}  -  \|F\|^{2}_{M} \right) \nonumber\\
&=&\frac{1}{2}\big(g(X , Y)\cdot \langle {\phi} , {\phi} \rangle_{M} - \langle X\lrcorner {\phi} , Y\lrcorner {\phi} \rangle_{M}\Big)+ \frac{1}{6} g(X,Y) \left(f^{2}  -  \|F \|^{2}_{M} \right) \nonumber\\
&=&\frac{1}{2}g(X, Y)\|\phi\|^{2}_{M}-\frac{1}{2}\langle X\lrcorner {\phi} , Y\lrcorner {\phi} \rangle_{M}+ \frac{1}{6} g(X,Y) \left( f^{2}  -  \|\phi \|^{2}_{M} \right) \nonumber\\
&=&- \frac{1}{2} \langle  X\lrcorner {\phi} , Y\lrcorner  {\phi} \rangle_{M}
+ \frac{1}{6}g(X,Y) \left(f^2  + 2 \| {\phi} \|^2_{M} \right). 
\end{eqnarray*}

 \noindent Thus, one can write
 \begin{equation}\label{Rricg}
 \Ric^{g}(X, Y)= \frac{1}{6}g(X,Y) \left(f^2  + 2 \| {\phi} \|^2_{M} \right)+q_{\phi}(X, Y),
 \end{equation}
 where $q_{\phi}(X, Y)$ is the symmetric bilinear form  $q_{\phi}(X, Y):= -\frac{1}{2}\langle  X\lrcorner {\phi} , Y\lrcorner  {\phi} \rangle_{M}$.

\noindent  Hence, motivated  by the results in this paragraph, we   introduce  the following definition:  
\begin{definition}
\textnormal{A   Riemannian 7-manifold    $(M^{7}, g, \phi)$  with  a special 3-form  $\phi$   is  called a} {\it  special gravitational    Einstein manifold} \textnormal{if  the pair  $(g,\phi)$ is  a  solution of  the supergravity Einstein  equation (\ref{Rricg}).}
\end{definition}
\begin{remark}\textnormal{
Note that a special gravitational    Einstein 7-manifold is not necessarily an Einstein manifold,   since $q_{\phi}$ is not necessarily a multiple of the metric  tensor $g$. In particular, (\ref{Rricg}) is an extension of the Einstein equation by   a  stress-energy tensor   associated  to the 3-form $\phi$.}
\end{remark}
\medskip
By Proposition \ref{nice1} (or Corollary \ref{special3}) and Proposition \ref{true1?}, it is obvious that the pair
  \[
 (g_{\mc{M}}=\wi{g}+g, \mc{F}^{4}= f\cdot \vol_{\wi{M}} +  {F}^{4}),
 \]
  where the closed 4-form $F^{4}$ is given by $F^{4}=\star_{7}\phi$ for some special 3-form $\phi$ on $M^{7}$, $g$ is a gravitational special Einstein metric and $\wi{g}$ a Lorentzian Einstein metric,  induces  solutions of eleven-dimensional  supergravity  on $\mc{M}^{10, 1}=\wi{M}^{3, 1}\times M^{7}$, which we shall call {\it  $(4, 7)$-decomposable solutions of eleven-dimensional  supergravity}. In this case,  $\mc{M}=\wi{M}\times M$ will be referred by the term {\it   $(4, 7)$-decomposable supergravity background}. We conclude that  
 \begin{corol}
   Any   $(4,7)$-decomposable  solution  $(\mathcal{M}^{10, 1}, g_{\mc{M}}, \mc{F})$  of eleven-dimensional supergravity, is  a product of Lorentzian Einstein 4-manifold  $(\wi{M}^{3, 1}, \wi{g})$ with negative Einstein constant   and  a gravitational  special  Einstein  7-manifold  $(M^7, g)$ with  special 3-form $\phi \in \Omega^3(M^7)$. In particular, the flux  4-form is given by $\mathcal{F} = f\cdot \vol_{\wi{M}} +F^{4}$ for some closed 4-form  $F^{4}:=\star_{7}\phi\in\Omega_{\rm cl}^{4}(M^{7})$ and
 some constant $f\in\bb{R}$.
\end{corol}


\subsection{Three basic types of $(4, 7)$-decomposable supergravity backgrounds}

We now consider  three basic  classes  of  special 3-forms on  Riemannian 7-manifolds, namely
\begin{itemize}
\item[(I)]  trivial 3-form, i.e. $\phi =0$ (and hence $F=0$) but  $f\neq 0$.
\item[(II)]   non-zero  harmonic   3-form, i.e.  $\phi \neq 0$, $f=0$.
\item[(III)]  non-harmonic  3-form, i.e.    $\phi \neq 0$, $f\neq 0$.
\end{itemize}
Let us examine the  construction of
   solutions   of the supergravity  Einstein  equation (\ref{Rricg})  for any of these types of special 3-forms (or equivalently the associated flux 4-forms, which will be referred by the same name), and  present   the corresponding special gravitational    Einstein manifolds $(M^{7}, g, \phi)$ and some    examples. 
    We begin with the first type.
   \begin{corol}
The  equation (\ref{Rricg})  for special 3-forms of Type I reduces to the   standard Einstein  equation, i.e.   $\Ric^{g} =(f^2/6) g$.
Consequently,  using  the  flux   4-form      $\mathcal{F}  =f\cdot\vol_{\wi{M}}$ we obtain a $(4, 7)$-decomposable supergravity  background, given by a product  of  a Lorentzian Einstein  4-manifold $(\wi{M}^{3,1}, \wi{g})$  with Einstein constant $-f^{2}/{3}$, and  a Riemannian  Einstein  7-manifold
$(M^7, g)$ with Einstein  constant $f^2/6$.
\end{corol}
\noindent Therefore,    flux forms of type $\mathcal{F}  =f\cdot\vol_{\wi{M}}$ with $f\in\bb{R}^{*}$,  induce  $(4, 7)$-decomposable supergravity backgrounds  by choosing a Lorentzian Einstein  4-manifold $(\wi{M}^{3, 1}, \wi{g})$ and a compact  Einstein  7-manifold $(M^7, g)$.

We treat now special 3-forms of Type II. In this case the flux form $\mc{F}$ is given by $\mc{F}=\star_{7}\phi=:F^{4}$. 
\begin{corol}
The    equation  (\ref{Rricg})    for   a   special harmonic  3-form  $\phi \neq 0$ on  $M^7$ of Type II,  reduces  to  the equation
\[
   \Ric^{g}= \frac13 \|\phi\|_{M}^2 g  -\frac{1}{2}q_{\phi},\quad   q_{\phi}(X,X) = ||X \lrcorner \phi||_{M}^2.
\]
Moreover, $(\wi{M}^{3, 1}, \wi{g})$  is Einstein with Einstein constant $-\|\phi\|^{2}/6$.
\end{corol}
\begin{remark}
\textnormal{Apriori,  we  may  consider  a    generic Type II special 3-form $\phi$. However,  such a 3-form is parallel     and  in Section \ref{G2struct}  we will show that    it   does   not induce      $(4, 7)$-decomposable supergravity backgrounds. }
\end{remark}
\begin{example}
\textnormal{Consider the  Riemannian product    $(M^7:= Q^3 \times P^4,  g = g_Q + g_P )$ between a 3-dimensional Riemannian manifold  $(Q^{3}, g_{Q})$ and a 4-dimensional  Riemannian manifold $(P^{4}, g_{P})$.  Assume that $M^{7}$ admits a special 3-form $\phi$, given by  $\phi:= \vol_Q$, where $\vol_{Q}$ is the  is   volume  3-form on  the  first factor, with  $\|\phi\|^{2}=\|\vol_{Q}\|^{2}=1$.
 Then  $\langle X\lrcorner\vol_{Q}, Y\lrcorner\vol_{Q}\rangle=g_{Q}(X, Y)$ for any $X, Y\in\Gamma(TM^{7})$. Hence  the  supergravity Einstein  equation becomes
\[
 \Ric^{g} = \frac{1}{3}g-\frac{1}{2}g_{Q},
 \]
and we conclude that $\Ric^{g_Q} = - \frac16 g_Q$ and $\Ric^{g_P} =\frac13 g_P$.
Therefore, the  manifolds  $Q, P$ must be Einstein manifolds  with Einstein  constant $-\frac16$ and  $\frac13$, respectively.
 Assume now  that  our initial metric  $g$ is  complete.    Then,    $Q$  is a complete  space  of constant  negative    curvature (i.e.  a quotient  $\mathbb{R}H^3/\Gamma$ of   the Lobachevski  space   $\mathbb{R}H^3$  by  a  lattice)  and     $P $ is  a  compact Einstein   4-manifold. Note  that    the manifold  $M^7$ is  compact if   $\Gamma$ is  a  co-compact lattice.
  So  we get  an example of  decomposable  supergravity background   of  Type II,  with   internal  space   $M^7 = Q^3 \times P^4$  and space-time
 any Lorentzian  Einstein  4-manifold $\wi{M}^{3, 1}$ with Einstein constant $-1/6$.}

\end{example}

The supergravity Einstein    equation  (\ref{Rricg})  for  a 7-manifold  $(M^7,g, \phi)$ where   $\phi$  is  a special  3-form of Type III, i.e.  a non-harmonic  3-form, remains unchanged.
In  the  next  section    we  study   this   case  under  the  assumption  that   $\phi$ is  a generic   3-form.

\section{$(4, 7)$-decomposable supergravity backgrounds of Type III associated  to $\G_2$-geometries}\label{G2struct}

Let us fix  the decomposable flux form $\mc{F}^{4}=f\cdot\vol_{\wi{M}}+\star_{7}\phi$, where $\phi=\star_{7}F^{4}$ is a special 3-form. Here  we  examine the situation  where $\phi$ is in addition    {\it generic}. To this end,  it will be  useful to refresh some notions of $\G_2$-structures (see \cite{Br1, Br2,  FKMS, Joy1} for more  details). 

\subsection{The Lie group $\G_2$ and $\G_2$-structures}\label{g2back}
  A 7-dimensional oriented Riemannian manifold  $(M^{7}, g)$ is called a $\G_2$-manifold whenever the structure group of its frame bundle $\SO(M, g)$ is  reduced   to   the   exceptional   compact    Lie  group   $\G_2\subset\SO_{7}$.    Recall that  the Lie   group  $\G_2$  has dimension 14 and  traditionally is defined as  the automorphism group of the octonion  algebra $\bb{O}$.  It  is also      defined  as  the stabilizer $G_{\omega}:=\{\al\in\Gl_{7}(\bb{R}) : \al\omega=\omega\}$ of a {\it generic} 3-form $\omega$  on $\bb{R}^{7} =  \mathrm{Im} \bb{O}$,  with  respect to  the  natural action  of   the  group $\Gl_7(\mathbb{R})$.
\begin{definition}
\textnormal{A 3-form $\omega \in  \bigwedge^{3}(\bb{R}^{7})^{*}$ is called}  {\it   generic} \textnormal{if its  stabilizer   $G_{\omega}$  in  $\Gl_7(\mathbb{R})$  is  the Lie group $\G_2$.}
\end{definition}
  A differential 3-form $\omega$  on  a 7-manifold  $M$ is   generic  if its  value  at   any point is  a generic  3-form.  Since   $\dim G_{\omega}=\dim\Gl_{7}(\bb{R})-\dim\bigwedge^{3}(\bb{R}^{7})^{*}=49-35=14$,
   the $\Gl_7(\mathbb{R})$-orbit  $\Omega^3_+$ of  a  generic 3-form    is open.   Another open    $\Gl_7(\mathbb{R})$-orbit is   the  orbit  $\Omega^3_{-}$  of  a 3-form with  stabilizer   the normal real form $\G_2^*$  of  $\G_2$,   which is  defined in terms of  splittable  octonions (see  \cite{Kath, Le2} for  more  details on   $\G_2^{*}$).  A generic 3-form  $\omega$  determines an  Euclidean metric $g$ by the rule
\[
 g(X, Y)\vol_{M}=-\frac{1}{6}(X\lrcorner \omega)\wedge (Y\lrcorner\omega)\wedge\omega,
 \]
 for any $X, Y\in\Gamma(TM)$. In terms of an appropriate  $g$-orthonormal basis   of co-vectors  $\{e^i\}$, $\omega$ has the form
 \begin{equation}\label{3form}
 \omega^{3}:=e^{127}+e^{347}+e^{567}+e^{135}-e^{245}-e^{146}-e^{236},
 \end{equation}
 where  $e^{ijk} = e^i \wedge e^j \wedge e^k$  denotes    the    wedge product of $e^{i}, e^{j}, e^{k}$.    A generic   3-form  $\omega$ on a 7-manifold  $M$  induces a $\G_2$-structure,  i.e. a subbundle of $\SO(M, g)$  which is  defined by  frames  $\{e_i\}$  with respect to  which  $\omega$ has   the   above   canonical   form (\ref{3form}), and  conversely any $\G_2$-structure   defines  a  generic   3-form.  So,  we may identify  a $\G_2$-structure  with the   corresponding   generic  3-form $\omega$.   We finally recall that  the existence of  a $\G_2$-structure   implies   the following restrictions on  the topology of  $M^{7}$:

\begin{prop}   \textnormal{(\cite[Prop.~3.2]{FKMS} or \cite{Joy1}[Prop. 3.6.2, Prop. 10.1.6])}      \label{topology}
The  existence   of a $\G_2$-structure  on a connected 7-dimensional manifold $M^{7}$   is equivalent to the vanishing of the first and the second Stiefel-Whitney classes of $M^{7}$ and hence equivalent to the existence of a  spin structure.
\end{prop}

 \begin{definition}
\textnormal{A $\G_2$-manifold  $(M^{7}, g, \omega)$  is called}
  \begin{itemize}
 \item  {\it parallel}, \textnormal{if $d\omega=0=d\star_{7}\omega$,}
 \item {\it weak $\G_2$}, \textnormal{if there exists $\lambda\in\bb{R}\backslash\{0\}$ such that $\dd\omega=\lambda\star_{7}\omega$ (and thus $\dd\star_{7}\omega=0)$,}
 \item {\it co-callibrated}, \textnormal{ if  $d\star_{7}\omega=0$.}
   \end{itemize}
 \end{definition}  
\noindent  	When     $({M}^{7}, g, \omega^{3})$ is a parallel $\G_2$-manifold, then   there exists a  $\nabla^{g}$-parallel spinor and hence  $({M}^{7}, g)$ is $\Ric^{g}$-flat \cite{Wang2}.   On the other hand,   the existence of  a  weak  $\G_2$-structure on a  compact 7-manifold $({M}^{7}, g)$ is equivalent to the existence of a spin structure carrying  a {\it real Killing spinor}  \cite{FKMS}, i.e. a non-trivial section $\varphi\in\Gamma(\Sigma^{g}M)$ of   the spinor bundle $\Sigma^{g}M$ over $M$ satisfying the equation   $\nabla^{g}_{X}\varphi=\lambda X\cdot\varphi$,
 for any $X\in\Gamma(TM)$ and some $0\neq \lambda\in\bb{R}$, where here $\nabla^{g}$ represents the spinorial Levi-Civita connection.   Thus, compact weak $\G_2$-manifolds are singled out by the fact that admit Killing spinors and hence are Einstein manifolds with positive scalar curvature, i.e. (see \cite{FKMS}),

\begin{equation}\label{weak2}
\Ric^{g} (X,Y) =\frac{3}{8} \lambda^2 g(X , Y), \, \quad \forall \ X, Y   \in   \Gamma(TM^{7}).
\end{equation}


\begin{remark}\textnormal{
Compact weak  $G_2$-manifolds  $(M^{7},\varphi, g)$   
admit an   equivalent  description in terms of   the metric   cone  $(\hat{M} = \mathbb{R}\times M^{7}, \hat{g}=  \dd r^2 + r^2g)$ over $M^{7}$. Since $(M^{7},\varphi,g)$ admits  Killing spinors, $(\hat{M}, \hat{g})$  admits   parallel   spinors  and   hence  has   holonomy  group $\mathrm{Hol}(\hat M) \subset  \mathrm{Spin}_7$. In particular, if  $(M^{7}, \varphi, g)$ is simply-connected and not isometric to the standard sphere, then  the  inclusions   $\Sp_2 \subset  \SU_4 \subset \Spin_7$  yield  the following    three natural  classes of
weak  $\G_2$-manifolds:
\begin{itemize}
\item If $\mathrm{Hol}(\hat M ) = \Sp_2 $,   then     $M^{7}$ is   called   3-Sasakian and  it  has a 3-dimensional  space  of  Killing  spinors.
\item If $\mathrm{Hol}(\hat M) = \SU_4 $,   then     $M^{7}$ is   called   Sasaki-Einstein manifold  and  it has a  2-dimensional  space of Killing  spinors.
\item If $\mathrm{Hol}(\hat M) = \Spin_7$,   then      $M^{7}$ is   called proper weak  $G_2$-manifold, with 1-dimensional  space of    Killing  spinors.
\end{itemize}}
\end{remark}

\subsection{$(4, 7)$-decomposable supergravity solutions  induced by weak  $\G_2$-structures}\label{sol}
Let us explain now how the above theory applies in supergravity equations and gives rise to special $(4, 7)$-decomposable supergravity backgrounds.
    Let   ${\phi}\equiv {\phi}^{3}$   be  a  {\it generic}   3-form on   $M^7$, i.e. assume that $(M^{7}, \phi)$ is a $\G_2$-manifold.
      We  will     normalise  $\phi$ such  that $\|{\phi} \|^2_{M} =\langle \phi, \phi\rangle_{M}=7$. Then  the identity $  \langle X \lrcorner {\phi} ,  Y \lrcorner {\phi} \rangle = 3 g(X, Y)$   holds, see  \cite{Br2} .
 Therefore,  equation (\ref{Rricg}) reduces to
\begin{equation}\label{hypothetic}
\Ric^{g} (X,Y)
= \frac{1}{6} \left(f^2  + 5 \right) g(X,Y),
\end{equation}
for any   $X, Y\in\Gamma(TM^{7})$.   Based on the previous description of weak $\G_2$-structures, Proposition \ref{nice1} (or Corollary \ref{special3}) and the relations (\ref{weak2}) and (\ref{hypothetic}),   we   check that when the associated flux 4-form $\mc{F} = f\cdot \vol_{\wi{M}} + \star_7  \phi  $   is a solution of the  supergravity Einstein   equations $(\mathscr{E})$,  then   it needs to hold  $f=\pm 2$.
Thus we obtain the following
\begin{theorem}\label{mainresult}
Let $\mc{M}^{10, 1}$ be the  oriented Lorentzian manifold given by  the product of a four-dimensional oriented Lorentzian manifold $(\wi{M}^{3, 1} , \wi{g})$ with volume form $\vol_{\wi{M}}$   and a seven-dimensional oriented manifold $M^{7}$ admitting a $\G_2$-structure  $\phi\in \Omega^{3}_{+}(M)$, such that  $\|{\phi} \|^2 = 7$. Define
\[
\mc{F}^{4}_\pm :=  \pm 2 \vol_{\wi{M}} + \star_7  {\phi}.
\]
 Then $(\mc{M}, g_{\mc{M}}=\wi{g}+g,\mc{F}^{4}_\pm)$, where $g$ is the  Riemannian metric on $M$ corresponding to $\phi$, gives rise to a pair of $(4, 7)$-decomposable supergravity backgrounds  if and only if  $({M}^{7}, \phi)$  is  a    weak $\G_2$-manifold
   and $(\wi{M}^{3, 1}, \wi{g})$ is Lorentz Einstein with negative Einstein constant $\Lambda:= -15/6$.   
\end{theorem}

 Let us also discuss  the case where the special 3-form $\phi$ is generic and of Type II, i.e.   $f=0$. Then, the closure condition and the Maxwell equation imply that $\phi$ is both closed and co-closed, so it induces  a parallel $\G_2$-structure. Therefore $({M}^{7}, g)$ must be Ricci-flat, and  by \eqref{hypothetic} we obtain
 \begin{prop}
 The 4-form $\mc{F}=F=\star_7  \phi$, where $\phi$ is a parallel $\G_2$-structure on $(M^{7}, g)$, i.e.\ $\phi$ is  a generic special 3-form of Type II, cannot satisfy the supergravity equations for the Lorentzian manifold $\mc{M}^{10, 1}=\wi{M}^{3, 1}\times M^{7}$, endowed with the induced product metric.
 \end{prop}

\section{Classification of 7-dimensional   homogeneous manifolds  of   a  compact  Lie group }\label{homg}
In this section we classify all compact almost effective homogeneous  7-manifolds $M^{7}=G/H$ of a compact connected Lie group $G$  (up  to  a   covering).  We  apply    this  to  the description  of    invariant  generic  (special) 3-forms, and  some invariant  non-generic  special 3-forms that solve  the Maxwell  equation.  In particular,  one can  separate the examination  of Type III invariant special 3-forms  into    the following  two   subclasses:
\begin{itemize}
\item Type III$\al$, i.e. $\phi:=\star_{7}F^{4}$ is an invariant {\it generic special 3-form} and thus  it induces a homogeneous co-callibrated    weak  $\G_2$-structure  on $M^{7}=G/H$.
\item  Type III$\be$, i.e. $\phi:=\star_{7}F^{4}$ is an invariant  {\it non-generic special 3-form}   on $M^{7}=G/H$.
\end{itemize}

\subsection{Classification  of  subalgebras  of   $\mathfrak{so}_7$}
 So, consider a seven-dimensional compact connected homogeneous Riemannian manifold $(M^{7}= G/H, g)$. 
 We  will  always  assume  that  the  action of  $G$  is   almost  effective,  that is  the  kernel of  effectivity $C=\{g \in G :  gx =x, \  \forall x \in  M  \}$
is  finite.   Let $\mathfrak{g} = \mathfrak{h} + \mathfrak{m}$  be     a  reductive  decomposition  of $\fr{g}$, such   that $\mathfrak{m}$ is identified  with  the tangent  space $T_oM^{7}$ of $M$, where  $o:= eH$.   The   isotropy   representation   $\chi : H \to \SO(\mathfrak{m})\cong\SO_7$  is   given  by $\chi(h)X   = \Ad_h X$, for any $h \in  H$ and $X \in  \mathfrak{m}$.  Almost  effectivity means   that  
the differential $\chi_* : \mathfrak{h} \to \mathfrak{so}(\mathfrak{m})$  of the isotropy  representation
    is   exact, i.e. $\ker(\chi_{*})=\{0\}$ (cf. \cite{Bes}).    Hence,  $\mathfrak{h}$   is isomorphic  to  the isotropy    subalgebra   $\chi_*(\mathfrak{h})\subset\fr{so}(\fr{m})=\mathfrak{so}_7$.

The classification  of almost effective homogeneous 7-manifolds of  a  compact  Lie group $G$  reduces  to  the description   of   all compact Lie  algebras   $\mathfrak{g}$   with  a  reductive  decomposition
 $\mathfrak{g} =\mathfrak{h} + \mathfrak{m}$, $\fr{m}= T_{o}M^{7}$,   whose isotropy representation  $\chi_* $   is exact and such  that $\mathfrak{h}=\chi_*(\mathfrak{h})$   generates   a  compact   subgroup $H$ of   a  compact Lie  group  $G$  with  the  Lie  algebra  $\mathfrak{g}$. This procedure   splits  into  two  simple  steps:
 \begin{itemize}
 \item Description  of  all   subalgebras $\mathfrak{h}$   of  the orthogonal Lie  algebra $\mathfrak{so}_7$.
\item  Description  of   all compact Lie   algebras $\mathfrak{g}$  which   contain $\mathfrak{h}$ as a  codimension 7  Lie subalgebra.
\end{itemize}
Since $\mathfrak{so}_7 = \mathfrak{b}_3$ is  a  rank  3  simple  Lie  algebra,  any   subalgebra  $\mathfrak{h}\subseteq\fr{so}_{7}$  is  a  compact Lie   algebra  of rank $r:=\rnk\fr{h}\leq 3$.  
 The  list   of  simple Lie algebras of    rank $\leq 3$  is   given below  (here the lower indices denote the rank, the upper indices denote the dimension): $\fr{a}_{1}^{3}=\fr{b}_{1}^{3}=\fr{c}_{1}^{3}, \ \fr{a}_{2}^{8}, \ \fr{a}_{3}^{15}=\fr{d}_{3}^{15},  \ \fr{b}_{2}^{10}=\fr{c}_{2}^{10}, \ \fr{g}_{2}^{14},  \  \fr{b}_{3}^{21}, \ \fr{c}_{3}^{21}$. Using it,  we   write  down    the    list  of  proper  semisimple  subalgebras   of  $\mathfrak{so}_7$: $  \mathfrak{so}_3, 2\mathfrak{so}_3,\, 3 \mathfrak{so}_3 = \mathfrak{so}_4 + \mathfrak{so}_3, \, \mathfrak{so}_5,\, \fr{su}_{4}=\mathfrak{so}_6,\, \mathfrak{su}_3$.
Calculating   the centralizer  of   these   subalgebras,  we  get the following    non-semisimple  proper   subalgebras  of  $\mathfrak{so}_7$:  $\mathfrak{u}_1, 2\mathfrak{u}_1, 3\mathfrak{u}_1, \mathfrak{so}_3 + \mathfrak{u}_1, \mathfrak{so}_3 +2\mathfrak{u}_1, \mathfrak{so}_5 + \mathfrak{u}_1, \mathfrak{u}_3$. Now, the    several non-conjugate   subalgebras  of  type $\mathfrak{so}_3$ can be described as follows. 
Let us denote by $V^k $ the irreducible  submodule of  real dimension  $k$  and  by  $\ell\mathbb{R}$  the  trivial $\ell$-dimensional  module.
Let $V^{3}:= \bb{R}^{3}$ be  the standard representation of $\fr{so}_{3}$ and  $V^4:= \mathbb{C}^2$ the standard representation of $\fr{su}_{2}$. 
  Recall  that there are two injective homomorphisms  $\fr{so}_{3}\to\fr{so}_{5}$ of $\fr{so}_{3}$ into $\fr{so}_{5}$, the standard one $A\mapsto{\rm diag}(A, 0, 0)$ and the embedding which corresponds to the unique 5-dimensional representation $V^{5}:=\bb{R}^{5}\cong \Sym^{2}_{0}(\bb{R}^{3})$.
 Similarly, we shall write $V^{7}:=\bb{R}^{7}\cong  \Sym^{3}_{0}(\bb{R}^{3})$ for the unique 7-dimensional irreducible representation of $\fr{so}_{3}$.

  Any   $\mathfrak{so}_3$ subalgebra of  $\mathfrak{so}_7$ is   given  by a  7-dimensional representation $\rho : \fr{so}_{3} \to \fr{so}_{7}\subset\fr{gl}(\bb{R}^{7})$ of
$\mathfrak{so}_3$, which   must be  a  direct  sum  of  the  irreducible  representations  $\mathbb{R}, V^3, V^4, V^5 , V^7$.  As before, we use upper indices to indicate   dimension  of irreducible  representations of   dimension  $>1$. Then,  up to conjugation in $\SO_7$,  we  get  the   following  description  of   subalgebras  of  $\mathfrak{so}_7$  isomorphic  to  $\mathfrak{so}_3$.
\begin{lemma} \label{dyn1} A subalgebra of  $\fr{so}_{3}$ type inside $\fr{so}_{7}$ coincides with one of the following:
\[
\begin{tabular}{ll}
$\al_1)$ \ $\fr{su}_{2}=\fr{so}_{3}^{4}$, such that $\bb{R}^{7}=V^{4}+3\bb{R}$, &  $\al_4)$ \ $\fr{so}_{3}^{(3, 3)}$, such that $\bb{R}^{7}=V^{3}+V^{3}+\bb{R}$, \\
$\al_2)$ \ $\fr{su}_{2}^{c}=\fr{so}_{3}^{(4, 3)}$, such that $\bb{R}^{7}=V^4 +V^{3}$, &  $\al_5)$ \ $\fr{so}_{3}^{5}$, such that $\bb{R}^{7}=V^5 + 2\mathbb{R}$,\\
 $\al_3)$ \ $\fr{so}_{3}^{3}$, such that $V^{3}+4\bb{R}$, &  $\al_6)$ \ $\fr{so}_{3}^{7}$, such that $\bb{R}^{7}=V^{7}$.
\end{tabular}
\]
\end{lemma}
 Since $\fr{so}_{3}^{4}=\fr{su}_{2}=\fr{sp}_{1}\subset\fr{so}_{5}=\fr{sp}_{2}$, the splitting of $\bb{R}^{7}$ in case $\al_1)$
 coincides with the isotropy representation of the 7-sphere $\Ss^{7}=\Sp_{2}/\Sp_{1}$ (see \cite{Zi, Le}).  On the other hand, the isotropy representation of the Stiefel manifold $\bb{V}_{5, 2}=\SO_{5}/\SO_{3}^{\rm st}$, where $\SO_3$ is embedded in $\SO_5$ diagonally, decomposes as $\bb{R}^{7}=V^{3}+V^{3}+\bb{R}$ and $V^{7}$ coincides with the isotropy representation of the    7-dimensional Berger sphere $B^{7}=\SO_{5}/\SO_3^{\rm ir}$ (see \cite{Br1}). Notice that $V^{5}$ coincides with the isotropy representation of the symmetric space $\SU_3/\SO_3$.

We treat now  subalgebras of rank 2. Up to conjugation in $\SO_7$ there are two  subalgebras of type $\fr{so}_{4}$ inside $\fr{so}_{7}$. The first corresponds to the standard embedding $A\to{\rm diag}(A, 0, 0, 0)$  and we write $\fr{so}_{4}=\fr{su}_{2}+\fr{su}_{2}'$, with decomposition $\bb{R}^{7}=V^{4}+3\bb{R}$. Notice that $\fr{su}_{2}$ and $\fr{su}_{2}'$ are conjugate in $\SO_7$. The second subalgebra of this type is denoted by $\fr{so}_{4}^{(4, 3)}=\fr{su}_{2}+\fr{su}_{2}^{c}$ with $\bb{R}^{7}=V^{4}+V^{3}$.     We proceed with  non-conjugate   subalgebras of type $\mathfrak{so}_3 + \mathfrak{u}_1$ inside $\fr{so}_{7}$.
\begin{lemma}\label{dyn2}
A subalgebra of $\fr{so}_{3}+\fr{u}_1$ type inside $\fr{so}_{7}$ coincides with one of the following:
\[
\begin{tabular}{ll}
$\be_1)$ \ $\fr{so}_{3}^{4}+\fr{u}_{1}^{2}=\fr{su}_2+\fr{u}_{1}^{2}$ with $\bb{R}^{7}=V^{4}+V^{2}+\bb{R}$, & $\be_5)$ \ $\fr{so}_{3}^{3}+\fr{u}_{1}^{2}$ with $\bb{R}^{7}=V^{3}+V^{2}+2\bb{R}$,\\
$\be_2)$ \ $\fr{so}_{3}^{4}+\fr{u}_1^{2, 2}=\fr{su}_2+\fr{u}_{1}^{2, 2}=:\fr{u}_2$ with $\bb{R}^{7}=V^{4}+3\bb{R}$, & $\be_6)$ \ $\fr{so}_{3}^{3}+\fr{u}_1^{2, 2}$ with $\bb{R}^{7}=V^{3}+V^{2}+V^{2}$, \\
$\be_3)$ \ $\fr{so}_{3}^{4}+\fr{u}_{1}^{2, 2, 2}=\fr{su}_{2}+\fr{u}_{1}^{2, 2, 2}$ with $\bb{R}^{7}=V^{4}+V^{2}+\bb{R}$, & $\be_7)$ \ $\fr{so}_{3}^{(3, 3)}+\fr{u}_1^{2, 2, 2}$ with $\bb{R}^{7}=V^{3}\otimes V^{2}+\bb{R}$, \\
$\be_4)$ \ $\fr{so}_{3}^{(4, 3)}+\fr{u}_{1}^{2, 2}=\fr{su}_{2}^{c}+\fr{u}_1^{2, 2}=:\fr{u}_{2}^{c}$ with $\bb{R}^{7}=V^{4}+V^{3}$,  &  $\be_8)$ \ $\fr{so}_{3}^{5}+\fr{u}_1^{2}$ with $\bb{R}^{7}=V^{5}+V^{2}$.
\end{tabular}
\]
Here $V^{2}:=\bb{C}^{1}$  states for the standard representation of $\fr{u}_{1}$. 
Notice that in the third case $\be_3)$ the Lie algebra $\fr{u}_1$ acts both on  $V^{4}$ and $V^{2}$, in the second case $\be_2)$  it acts on $V^{4}$ and in the first case $\be_1)$ it acts only on $V^{2}$.
\end{lemma}
 \begin{proof}
We use  Lemma \ref{dyn1}  and compute   the   centralizers  of all  subalgebras  inside $\fr{so}_{7}$ of type $\fr{so}_{3}$. We see that
\[
\begin{tabular}{l l l}
$C_{\mathfrak{so}_7}(\mathfrak{so}^3_3 )=\mathfrak{so}_4$, &  $C_{\mathfrak{so}_7}(\mathfrak{su}_2) = \mathfrak{su}'_2 + \mathfrak{so}_4$, &   $C_{\mathfrak{so}_7}(\mathfrak{su}^c_2 )=  \mathfrak{su}_2'$, \\
   $ C_{\mathfrak{so}_7} (\mathfrak{so}_3^{(3,3)}) = \mathfrak{u}_1^{2, 2, 2}$, &   $C_{\mathfrak{so}_7}  (\mathfrak{so}_3^5) = \mathfrak{u}_1^{2}$, & $C_{\mathfrak{so}_7}   (\mathfrak{so}_3^7) =\{0\}$.
   \end{tabular}
\]
Hence we need to exclude  $\fr{so}_{3}^{7}+\fr{u}_{1}$ and our claim follows by considering the several possible actions of $\fr{u}_1$ (the case arising by the decomposition $\bb{R}^{7}=V^{5}+2\bb{R}$ cannot exist  due to the $\fr{u}_1$-action). 
\end{proof}


Concerning   subalgebras of rank 3, we  remark that  $\fr{so}_{4}+\fr{so}_{2}=\fr{su}_{2}+\fr{su}_{2}'+\fr{u}_1$ belongs to $\fr{so}_{7}$, but this is not true for  the direct sum $\fr{so}_{4}^{(4, 3)}+\fr{so}_{2}=\fr{su}_{2}+\fr{su}_{2}^{c}+\fr{u}_{1}$. Indeed, in the first case one computes $C_{\fr{so}_{7}}(\fr{so}_{4})=\fr{su}_{2}$, while the centralizer of $\fr{so}_{4}^{(4, 3)}$ is trivial, i.e. $C_{\fr{so}_{7}}(\fr{so}_{4}^{(4, 3)})=\{0\}$.  Let us  summarise   all the  results   (including Lemmas \ref{dyn1}, \ref{dyn2}) with  some   more information    in   Table 1. 

 \begin{center}
 {\bf{Table 1.}}  The  Lie subalgebras of $\fr{so}_{7}=\fr{b}_{3}$ 
  \[
{}\begin{tabular} {l | l | l |  l  }
$r=\rnk \fr{h}$ & $\fr{h}=\fr{h}^{d}$ & $\fr{g}^{d+7}$ &   $\fr{h}$-decomposition of $\bb{R}^{7}$ \\
\thickline
$r=0$ & $\fr{h}=\text{trivial}$ & $\fr{g}^{7}$   &   $                $  \\
\hline
$r=1$ & $\fr{u}_{1}$ & $\fr{g}^{8}$  & $\bb{R}^{7}=V^{2}+5\bb{R}$  \\
          & $\fr{u}_{1}$ & $\fr{g}^{8}$ & $\bb{R}^{7}=2V^{2}+3\bb{R}$ \\
          & $\fr{u}_{1}$ & $\fr{g}^{8}$ & $\bb{R}^{7}=3V^{2}+\bb{R}$ \\
           & $\fr{su}_{2}=\fr{so}_{3}^{4}$  &  $\fr{g}^{10}$  & $\bb{R}^{7}=V^{4}+3\bb{R}$ \\
           & $\fr{su}_{2}^{c}$  & $\fr{g}^{10}$ & $\bb{R}^{7}=V^{4}+V^{3}$ \\
           & $\fr{so}_{3}^{3}$ &  $\fr{g}^{10}$ &  $\bb{R}^{7}=V^{3}+4\bb{R}$  \\
           & $\fr{so}_{3}^{5}$ & $\fr{g}^{10}$ & $\bb{R}^{7}=V^{5}+2\bb{R}$ \\
           & $\fr{so}_{3}^{(3, 3)}$ & $\fr{g}^{10}$ &  $\bb{R}^{7}=V^{3}+V^{3}+\bb{R}$ \\
           & $\fr{so}_{3}^{7}$, &  $\fr{g}^{10}$  & $\bb{R}^{7}=V^{7}$\\
           \hline
$r=2$ & $2\fr{u}_{1}=\diag(\fr{u}_{1}+\fr{u}_{1})+\fr{u}'_{1}$ & $\fr{g}^{9}$ & $\bb{R}^{7}=V^{2}\otimes \mathbb{R}^2 + (V')^2+\bb{R}$   \\
                          & $\fr{so}_{3}^{4}+\fr{u}_{1}^{2}=\fr{su}_{2}+\fr{u}_{1}^{2}$ & $\fr{g}^{11}$ & $\bb{R}^{7}=V^{4}+V^{2}+\bb{R}$ \\
           &  $\fr{u}_{2}:=\fr{so}_{3}^{4}+\fr{u}_{1}^{2, 2}=\fr{su}_{2}+\fr{u}_{1}^{2, 2}$ & $\fr{g}^{11}$ & $\bb{R}^{7}=V^{4}+3\bb{R}$\\
           & $\fr{so}_{3}^{4}+\fr{u}_{1}^{2, 2, 2}$ & $\fr{g}^{11}$ & $\bb{R}^{7}=V^{4}+V^{2}+\bb{R}$ \\
           &  $\fr{u}_{2}^{c}:=\fr{so}_{3}^{(4, 3)}+\fr{u}_{1}^{2, 2}=\fr{su}_{2}^{c}+\fr{u}_{1}^{2, 2}$ & $\fr{g}^{11}$  &  $\bb{R}^{7}=V^{4}+V^{3}$ \\
             & $\fr{so}_{3}^{3}+\fr{u}_{1}^{2}$ & $\fr{g}^{11}$ & $\bb{R}^{7}=V^{3}+V^{2}+2\bb{R}$\\
              & $\fr{so}_{3}^{3}+\fr{u}_{1}^{2, 2} $ & $\fr{g}^{11}$ & $\bb{R}^{7}=V^{3}+V^{2}+V^{2}$\\
           & $\fr{so}_{3}^{(3, 3)}+\fr{u}_{1}^{2, 2, 2}$ & $\fr{g}^{11}$ & $\bb{R}^{7}=V^{3}\otimes V^2+\bb{R}$ \\
           & $\fr{so}_{3}^{5}+\fr{u}_{1}^{2}$ & $\fr{g}^{11}$ & $\bb{R}^{7}=V^{5}+V^{2}$ \\
            & $ \fr{so}_{4}=  \fr{su}_{2}+\fr{su}'_{2}$  & $\fr{g}^{13}$ & $\bb{R}^{7}=V^{4}+3\bb{R}$  \\
           & $\fr{so}_{4}^{(4, 3)}=\fr{su}_{2}+\fr{su}_{2}^{c}$ & $\fr{g}^{13}$ &    $\bb{R}^{7}=V^{4}+V^{3}$ \\
           & $\fr{su}_{3}$ & $\fr{g}^{15}$ & $\bb{R}^{7}=V^{6}+\bb{R}$  \\
           & $\fr{so}_{5}=\fr{sp}_{2}$ & $\fr{g}^{17}$ & $\bb{R}^{7}=V^{5}+2\bb{R}$\\
           & $\fr{g}_{2}$ &     $\fr{g}^{21}$ & $\bb{R}^{7}=V^{7}$\\
           \hline
$r=3$ & $3\fr{u}_{1}$ & $\fr{g}^{10}$  & $\bb{R}^{7}=3V^{2}+\bb{R}$ \\
          & $2\fr{u}_{1}+\fr{su}_{2}=\fr{u}_{2}+\fr{u}_{1}$ & $\fr{g}^{12}$ & $\bb{R}^{7}=V^{4}+V^{2}+\bb{R}$  \\
          & $\fr{so}_{4}+\fr{so}_{2} = \fr{su}_{2}+\fr{su}'_{2} + \fr{u}_{1}$ & $\fr{g}^{14}$ & $\bb{R}^{7}=V^{4}+V^{2}+\bb{R}$ \\
                   & $\fr{u}_{3}$ & $\fr{g}^{16}$ & $\bb{R}^{7}=V^{6}+\bb{R}$ \\
         & $\fr{su}_{2}+\fr{su}'_{2}+\fr{so}_{3}=\fr{so}_{4}+\fr{so}_{3}$ & $\fr{g}^{16}$ & $\bb{R}^{7}=V^{4}+V^{3}$ \\
          & $\fr{so}_{5}  + \fr{u}_{1} =\fr{sp}_{2} =\fr{so}_{2}$ & $\fr{g}^{18}$  & $\bb{R}^{7}=V^{5}+V^{2}$\\
          &  $\fr{so}_{6}$ & $\fr{g}^{22}$ & $\bb{R}^{7}=V^{6}+\bb{R}$ \\
          &  $\fr{so}_{7}$ &  $\fr{g}^{28}=\fr{d}_{4}$ & $\bb{R}^{7}=V^7$\\
          \thickline
\end{tabular}
\]
\end{center}


\subsection{Classification  of almost-effective compact homogeneous 7-manifolds}
  Now,   the  classification   of     almost  effective   homogeneous  7-manifolds $M^7 = G/H$  of   a compact  Lie group  $G$,  reduces  to  an  enumeration  of   all compact Lie   algebras $\fr{g}=\mathfrak{g}^{d + 7}$ of   dimension $d +7$, which contain    a subalgebra  $\fr{h}=\mathfrak{h}^d$  from Table 1 and have as  reductive  decomposition  $\mathfrak{g}^{d+7}   = \mathfrak{h}^d + \mathfrak{m}$, one of the indicated isotropy representations.   We present all   such homogeneous  7-manifolds  in  Table  2, but initially   it  is  convenient  to use Lemma \ref{dyn2}  and present a proof for the  almost effective cosets $M^{7}=G^{d+7}/H^{d}$ whose isotropy subalgebra  $\fr{h}^{d}\subset\fr{so}_7$ is of type $\fr{so}_{3}+\fr{u}_1$ (and hence $d=4$).  We mention that in Table 2   we omit the details for most of the embeddings $\fr{h}\subset\fr{so}_{7}$ which do not give rise to some almost effective coset and  use the following notation:  For  a given direct product   $M=G/H \times{\rm T}^{k}$ of a homogeneous space $G/H$ (whose isotropy subgroup is given by  $H = H' \times {\rm T}^{\ell}$) with a torus ${\rm T}^{k}$,  we  shall denote   by
  $M_{\psi}=G/H \wi{\times}{\rm T}^{k}   $   the
{\it  twisted product} $M_{\psi} = G/H^{\psi}$,  defined   by  a homomorphism $\psi : H = H' \times {\rm T}^{\ell}\to{\rm T}^{k}$,
where  $H^{\psi}:= \{ (h, \psi(h)) : h\in H \} \subset H \times {\rm T}^{k}$. It is remarkable that several cosets $M^{7}=G/H$ is of this type. 
\begin{prop}
Let $M^{7}=G^{11}/H^{4}$ be an almost effective homogeneous 7-manifold of an eleven-dimensional compact Lie group $G$, whose stability subalgebra   $\fr{h}\equiv\fr{h}^{4}$ is of type $\fr{so}_{3}+\fr{u}_1$.    Then $M$ is diffeomorphic to one of the cosets appearing in Table 2, case $d=4$.
\end{prop}
\begin{proof}
It is useful to split the examination of   compact Lie algebras $\fr{g}^{11}$ into two main cases:  

\noindent  {\bf  Case A: $\fr{g}^{11}$ is semisimple.}  Let us assume that $\fr{g}^{11}$ is semisimple, i.e.    $\mathfrak{g}^{11}=[\fr{g}^{11}, \fr{g}^{11}]$.  The  only semisimple  eleven-dimensional Lie  algebra is  the direct sum $\mathfrak{a}_1+\mathfrak{a}_2$, hence we set $\fr{g}^{11}=\mathfrak{so}_3+\mathfrak{su}_3  =  \mathfrak{su}_2+\mathfrak{su}_3$.   The  only subalgebras  of  type $\fr{so}_{3}$ inside  $\mathfrak{su}_3$ are the subalgebras   $\mathfrak{su}_2 = \mathfrak{so}_3^4$ and  $\mathfrak{so}_3^5$,  whose  centralizer  in  $\mathfrak{su}_3$   is $\mathfrak{u}_1$ and $\{0\}$, respectively. Therefore,  the following cases appear:

{\bf 1)} \ If $\mathfrak{su}_2 \subset \mathfrak{su}_3$,    then  $\mathfrak{h} = \mathfrak{so}_3^4 + \mathfrak{u}_1^{2,2}=\mathfrak{u}_2$. This  gives rise to the homogeneous space $M=\mathbb{C}P^2 \times\Ss^{3}= (\SU_3/\U_2)\times\SU_2$ with isotropy representation $\bb{R}^{7}=V^{4}+3\bb{R}$. 

 {\bf 2)}  If $\mathfrak{so}_3 \subset \mathfrak{su}_3$ and $\mathfrak{u}_1 \subset \mathfrak{so}_3 \subset \mathfrak{su}_2 + \mathfrak{su}_3$, then we deduce that there are two desired subalgebras of type $\fr{so}_3+\fr{u}_1$. The first one is given by $\mathfrak{h} = \mathfrak{so}_3^4 + \mathfrak{u}_1^2$ and induces the coset  $M = \Ss^2 \times \Ss^5 = (\SU_2/\U_1) \times (\SU_3/\SU_2)$, whose isotropy representation decomposes as $\bb{R}^{7}=V^{2}+V^{4}+\bb{R}$. 
    The second one coincides with $\mathfrak{h} = \mathfrak{so}_3^{5} + \mathfrak{u}_1^2$ with corresponding coset $M=(\SU_2/\U_1) \times (\SU_3/\SO_3)$. Here, the isotropy representation is given by $\bb{R}^{7}=V^{2}+V^{5}$. 

 {\bf 3)}  If $\mathfrak{so}_3 \subset \mathfrak{su}_3$ but $\mathfrak{u}_1\nsubseteq\mathfrak{so}_3$,  then
   $\mathfrak{h} = \mathfrak{su}_2 + \mathfrak{u}_1^{2,2,2}$    where   $\mathfrak{su}_2=\fr{so}_{3}^{4}$ is  the   standard   subgroup of   $\mathfrak{su}_3$  and   $\mathfrak{u}_1^{2,2,2} = \Delta \mathfrak{u}_1 $ is   the   diagonal  subgroup  of  $ \mathfrak{u}_1 + \mathfrak{u}_1 \subset  \mathfrak{su}_2+ \mathfrak{su}_3$.
  Then we    get  the  homogeneous space $M = (\SU_3\times\SU_2)/(\SU_2\times\U_1)=\big((\SU_3/\SU_2)\times\SU_2\big)/\Delta\U_1$, whose isotropy representation decomposes as follows: $\bb{R}^{7}=V^{4}+V^{2}+\bb{R}$.  Usually, the embedding of $\Delta\fr{u}_{1}$ in  $\fr{u}_1+\fr{u}_1$ is indicated by two parameters $a, b$ and it is classical to denote these manifolds by $N_{a, b}$. 

 {\bf 4)} If $\mathfrak{su}_2\nsubseteq\mathfrak{su}_3$,  then $\fr{h}=\fr{so}_{3}^{(4, 3)}+\fr{u}_{1}^{2, 2}=\fr{su}_{2}^{c}+\fr{u}_{1}^{2, 2}=\fr{u}_{2}^{c}$, where we identify $\fr{su}_{2}^{c}$ with     the  diagonal  subalgebra  $\Delta \mathfrak{su}_2$  of $\mathfrak{su}_2 \oplus  \mathfrak{su_2}' \subset \mathfrak{su}_2 \oplus  \mathfrak{su_3}$,
  and $\mathfrak{u}_1 = \mathfrak{u}_1^{2,2}$ with    the centralizer  of $\mathfrak{su}_2'$  in  $\mathfrak{su}_3$.  This gives rise to the so-called exceptional Allof-Wallach spaces $W_{1, 1}=(\SU_{3}\times\SU_{2})/(\SU_{2}^{c}\times\U_1)$, with  isotropy representation   $\bb{R}^{7}=V^{4}+V^{3}$. Note that here  the Lie group $\SU_{2}^{c}$ can be viewed as the normalizer of $\Delta\SU_2$ inside  $\SU_3\times\SU_2$.

  In order to complete Case A, we need to show that the  subalgebra   $\fr{h}=\fr{so}^{3}_{3}+\fr{u}_{1}^{2, 2}$ does not  induce some almost effective homogeneous 7-manifold. Indeed, since $\bb{R}^{7}=V^{3}+V^{2}+V^{2}$, the eleven-dimensional Lie algebra $\fr{g}^{11}$ must be without center, and thus we get  $\fr{g}^{11}=\fr{su}_{3}+\fr{su}_2$.  However, it must be $\fr{so}_{3}^{3}\subset\fr{su}_{3}$ but only $\fr{su}_2, \fr{so}_{3}^{5}$ have non-trivial centralizer inside $\fr{su}_3$ and our claim follows. 


    \smallskip
 \noindent {\bf  Case B: $\fr{g}^{11}$ is non-semisimple.}  Assume now that $\fr{g}^{11}$ is non-semisimple. Then the dimension of the center $Z(\fr{g}^{11})$ must satisfy $1\leq \dim Z(\fr{g}^{11})\leq 3$. Hence we need to consider three cases:

 {\bf  1)} \  $\dim Z(\fr{g}^{11})=1$. The unique candidate of a Lie algebra of type $\fr{g}^{11}=\fr{s}+\fr{u}_1$ with $\fr{s}$ simple, is the Lie algebra $\fr{g}^{11}=\fr{so}_{5}+\fr{u}_1=\fr{sp}_{2}+\fr{u}_1$. Inside $\fr{so}_{5}$ the  $\mathfrak{so}_3$-subalgebras   $\mathfrak{so}_3^{(3, 3)}$
  and     $\mathfrak{su}_2  \subset \mathfrak{u}_2$ have non trivial centralizer and the same holds for  $\fr{su}^{c}_{2}=\fr{so}^{(4, 3)}_{3}$ inside $\fr{sp}_{2}$. Hence, in this case we find  the following  subalgebras of type $\fr{so}_{3}+\fr{u}_1$ which induce  almost effective homogeneous 7-manifolds:\\
 $\bullet$ $\fr{h}=\fr{so}^{4}_{3}+\fr{u}_1^{2}$, with corresponding coset  $M=(\SO_5/\U_2)\wi{\times}{\rm S}^{1}=\bb{C}P^{3}\wi{\times}\Ss^{1}$ and $\bb{R}^{7}=V^{4}+V^{2}+\bb{R}$. \\  
 $\bullet$ $\fr{h}=\fr{so}_{3}^{(4, 3)}+\fr{u}_{1}^{2, 2}=\fr{u}_{2}^{c}$, which defines the squashed 7-sphere $\Ss^{7}= (\Sp_{2}\times\U_{1})/(\Sp_{1}\times\Delta\U_{1})$. Here, the isotropy representation is such that $\bb{R}^{7}=V^{4}+V^{3}$.  \\
$\bullet$ $\fr{h}=\fr{so}_{3}^{(3, 3)}+\fr{u}_{1}^{2, 2, 2}$, which induces the twisted product ${\rm Gr}_2(\mathbb{R}^5) \wi{\times}\Ss^{1}= (\SO_5/ \SO_3 \times \SO_2)\wi{\times}\Ss^{1}$, where   ${\rm Gr}_2(\mathbb{R}^5)$ is a Grassmann manifold. In this case the isotropy representation decomposes by $\bb{R}^{7}=(V^{3}\otimes V^{2})+\bb{R}$, where we identify the irreducible representation $V^{3}\otimes V^{2}$ with the isotropy representation of the six-dimensional symmetric space ${\rm Gr}_2(\mathbb{R}^5)$.

 {\bf 2)} \  $\dim Z(\fr{g}^{11})=2$.  Then  $\fr{g}^{11}=3\fr{so}_3+2\fr{u}_1=3\fr{su}_2+2\fr{u}_1$ and $\fr{h}=\fr{so}_{3}^{3}+\fr{u}_{1}^{2}$. In this case we obtain  the space $M=(\SO_4/\SO_3)\times(\SU_2/\U_1)\wi{\times}{\rm T}^{2}=\Ss^{3}\times\Ss^{2}\wi{\times}{\rm T}^{2}$, with $\bb{R}^{7}=V^{3}+V^{2}+2\bb{R}$.

 {\bf 3)}  $\dim Z(\fr{g}^{11})=3$.  Then $\fr{g}^{11}=\fr{su}_{3}+3\fr{u}_1$ and the isotropy subalgebra $\fr{h}$ must be $\fr{so}_{3}^{4}+\fr{u}_{1}^{2, 2}=\fr{u}_{2}$. Thus here we get the coset $M=\bb{C}P^{2}\wi{\times}{\rm T}^{3}$, whose isotropy representation decomposes as $\bb{R}^{7}=V^{4}+3\bb{R}$.
                 \end{proof}

{\small
\begin{center}
 {\bf{Table 2.}}    Compact  almost effective homogeneous 7-manifolds $M^{7}=G/H$.
 \[
\begin{tabular} {r | l | l | l | c | c | l }
$d$  & $\fr{h}$ & $\fr{g}\equiv\fr{g}^{d+7}$ &   $M^{7}=G^{d+7}/H^{d}$ & $\G_2^{\rm inv}$ & $\text{np}\G_2^{\rm inv}$ & $\mc{E}_{\rm inv}$ \\
\thickline
$d=0$ & $\{0\}$                                   & $7\fr{u}_{1}$                                                & ${\rm T}^{7}$ & $\checkmark$ & $\times$ & $\times$ \\
       &                                              & $\fr{su}_2+ 4\fr{u}_{1}$                        & $\SU_{2}\times {\rm T}^{4}= \Ss^{3}\times{\rm T}^{4}$ & $\checkmark$ & $\times$ & $\times$ \\
       &                                              & $2\fr{su}_2+ \fr{u}_{1}$                            & $\SU_2\times\SU_2\times {\rm T}^{1}= \Ss^{3}\times\Ss^{3}\times\Ss^{1}$ & $\checkmark$ & $\times$ & $\times$ \\
       \thickline
$d=1$ & $\fr{u}_{1}$                            & $\fr{su}_3$                                                     & $W_{k, l}:=\displaystyle\frac{\SU_{3}}{\U_{1}^{k, l}}$   & $\checkmark$ & $\checkmark$ & 2  \\
 &    & & $(k, l\in\bb{Z}_{\geq 0}, \ k\geq l\geq 0, \ kl>1)$ & & \\

       &                                             & $2\fr{su}_2+ 2\fr{u}_{1}$                    & $\bb{V}_{4, 2}\wi{\times}{\rm T}^{2}= \displaystyle\frac{\SU_{2}\times \SU_{2}}{\U_1} \wi{\times} {\rm T}^{2}=\frac{\SO_{4}}{\SO_{2}}\wi{\times}{\rm T}^{2}$ & $\checkmark$ & $\times$ & $\times$ \\
       &                                             &  $\fr{su}_2+ 5\fr{u}_{1}$                       & $\bb{C}P^{1}\wi{\times}{\rm T}^{5}=\Ss^{2}\wi{\times}{\rm T}^{5} = \displaystyle\frac{\SU_{2}}{\U_{1}}\wi{\times}  {\rm T}^{5} $ & $\times$ & $\times$  & $\times$  \\
       \thickline
$d=2$ &$2\fr{u}_{1}$                           & $\fr{su}_2+6\fr{u}_{1}$                         &  no almost effective coset & $\times$ & $\times$    & $\times$    \\
       &                                             & $2\fr{su}_2+3\fr{u}_{1}$                     &  $\displaystyle\frac{\SU_{2}}{\U_1}\times\frac{\SU_{2}}{\U_1}\wi{\times} {\rm T}^{3}=\Ss^{2}\times\Ss^{2}\wi{\times} {\rm T}^{3}$  & $\times$ & $\times$ & $\times$        \\
       &                                             &  $3\fr{su}_2$                                                &  $M_{a, b, c}=\displaystyle\frac{\SU_{2}\times\SU_{2}\times\SU_{2}}{\U_1\times\U_1}$ & $\checkmark$ & $\checkmark$   & 1 or 2 \\ 
  &    & & $(a\geq b\geq c\geq 0$, $a>0$, ${\rm gcd}(a, b, c)=1)$ & & \\
       &                                             &    $\fr{su}_3+\fr{u}_{1}$                             & $\bb{F}_{1, 2}\wi{\times} {\rm S}^{1} = \displaystyle\frac{\SU_{3}}{{\rm T}_{\rm max}}\wi{\times}\Ss^{1}$ & $\checkmark$ & $\times$  & $\times$       \\
       &                                              &                                                                    & $W_{k, l}:=\displaystyle\frac{\SU_{3}}{\U_{1}^{k, l}}$ \  $(k, l \ \text{arbitary})$ & $\checkmark$ & $\checkmark$  & 2     \\
       \thickline
$d=3$ & $\al_{1})$   $\fr{su}_{2}=\fr{so}_{3}^{4}$          & $\fr{su}_2+ 7\fr{u}_{1}$                          &  no almost effective coset & $\times$ & $\times$  & $\times$      \\ 
          &                      & $\fr{sp}_2$                                           & $\Ss^{7}_{V^{4}+3\bb{R}} = \ \displaystyle\frac{\Sp_{2}}{\Sp_{1}}$ & $\checkmark$ & $\checkmark$ & 2    \\
          &               & $\fr{su}_3+2\fr{u}_{1}$                          & $\Ss^{5}_{V^{4}+\bb{R}}\times {\rm T}^{2}= \displaystyle\frac{\SU_{3}}{\SU_{2}}\times {\rm T}^{2}$ & $\checkmark$ & $\times$   & $\times$      \\
          & $\al_2)$ $\fr{su}_{2}^{c}=\fr{so}_{3}^{(4, 3)}$ & $\fr{g}^{10}\supset\fr{su}_{2}^{c}$ &  no almost effective coset & $\times$ & $\times$       & $\times$  \\
                 &    $\al_{3})$   $\fr{so}_{3}^{3}$                                 & $2\fr{su}_2+ 4\fr{u}_{1}$                    &   $ \Ss^{3}\times {\rm T}^{4} =\displaystyle\frac{\SO_{4}}{\SO_{3}}\times{\rm T}^{4}=\displaystyle\frac{\SU_{2}\times \SU_{2}}{\Delta \SU_{2}}\times {\rm T}^{4}$ & $\checkmark$ & $\times$   & $\times$     \\
        &       $\al_{4})$  $\fr{so}_{3}^{(3, 3)}$                                         & $3\fr{su}_2+\fr{u}_{1}$                         &   $\displaystyle\frac{\SO_{3}\times\SO_{3}\times\SO_{3}}{\Delta\SO_{3}}\times\Ss^{1} =\Ss^{3}\times\Ss^{3}\times\Ss^{1}$ & $\checkmark$ & $\times$   & $\times$     \\
            &         & $\fr{so}_5$                   & $\bb{V}_{5, 2}=   \SO_{5}/\SO_{3}^{\rm st}$ & $\checkmark$ & $\checkmark$    & 1   \\
        &       $\al_{5})$  $\fr{so}_{3}^{5}$                                      & $\fr{su}_3+2\fr{u}_{1}$                          & $Q_{1}^{7}=\displaystyle\frac{\SU_{3}}{\SO_{3}}\times {\rm T}^{2}$ & $\times$ & $\times$  & $\times$      \\

        &      $\al_{6})$  $\fr{so}_{3}^{7}$  & $\fr{so}_5$                                  & $B^{7}= \SO_{5}/\SO_{3}^{\rm ir}$ & $\checkmark$ & $\checkmark$   & 1 $g_{\rm irr}$   \\
              &     $3\fr{u}_{1}$            & $3\fr{su}_2+ \fr{u}_{1}$                          & $\Ss^{2}\times\Ss^{2}\times\Ss^{2}\widetilde{\times}\Ss^{1}$ & $\times$ & $\times$ & $\times$       \\ 
              \thickline
$d=4$ & $\be_1)$    $\fr{so}_{3}^{4}+\fr{u}_{1}^{2}$   &   $\fr{su}_3+\fr{su}_2$                        &   $\Ss^{5}_{V^{4}+\bb{R}}\times\Ss^{2}=\displaystyle\frac{\SU_3}{\SU_2}\times\frac{\SU_2}{\U_1}$  & $\times$ & $\times$  & 1     $g_{\rm sym}$    \\
   &                                            & $\fr{so}_{5}+\fr{u}_{1}$                         &  $ \bb{C}P^{3}\wi{\times}\Ss^{1}= \displaystyle\frac{\SO_{5}}{\U_{2}}\wi{\times}{\Ss}^{1}=\displaystyle\frac{\Sp_2}{\Sp_{1}\times\U_1}\wi{\times}{\Ss}^{1}$ & $\checkmark$ & $\times$ & $\times$        \\ 
                 & $\be_2)$ $\fr{so}_{3}^{4}+\fr{u}_{1}^{2, 2}=\fr{u}_2$ &           $\fr{su}_3+\fr{su}_2$                        &   $\bb{C}P^{2}\times\Ss^{3}=  \displaystyle\frac{\SU_{3}}{\U_{2}}\times\SU_{2}$  & $\times$ & $\times$   & 1 $g_{\rm sym}$    \\
                  &                                             & $\fr{su}_3+3\fr{u}_{1}$                          &  $\bb{C}P^{2}\wi{\times}{\rm T}^{3} \ = \  \displaystyle\frac{\SU_{3}}{\U_{2}}\wi{\times} {\rm T}^{3}$ & $\times$ & $\times$ & $\times$        \\
                       & $\be_3)$ $\fr{so}_{3}^{4}+\fr{u}_{1}^{2, 2, 2}$ &  $\fr{su}_3+\fr{su}_2$                        &   $N_{a, b}=\displaystyle\frac{\SU_{3}\times\SU_{2}}{\SU_{2}\times\U_1}=\Big(\displaystyle\frac{\SU_3}{\SU_2}\times\SU_2\Big)/\Delta\U_1$  & $\checkmark$ & $\checkmark$  & 1     \\

       &    $\be_4)$    $\fr{su}_{2}^{c}+\fr{u}_{1}^{2, 2}=\fr{u}_{2}^{c}$                                         &   $\fr{su}_3+\fr{su}_2$               &                                                                                                               $W_{1, 1}= \ \displaystyle\frac{\SU_{3}\times\SU_{2}}{\SU_{2}^{c}\times\U_1}$ & $\checkmark$ & $\checkmark$  & 2      \\ 

      &                                             & $\fr{sp}_{2}+\fr{u}_{1}$        &  $\Ss^{7}_{V^{4}+V^{3}} = \ \displaystyle\frac{\Sp_{2}\times\U_{1}}{\Sp_{1}\times\Delta\U_{1}}$ & $\checkmark$ & $\checkmark$ & 2       \\

             & $\be_5)$ $\fr{so}_{3}^{3}+\fr{u}_{1}^{2}$                & $3\fr{su}_{2}+2\fr{u}_{1}$ & $\displaystyle\frac{\SO_4}{\SO_3}\times\frac{\SU_2}{\U_1}\wi{\times}{\rm T}^{2}=\Ss^{3}\times\Ss^{2}\wi{\times}{\rm T}^{2}$ & $\times$ & $\times$       & $\times$  \\
             & $\be_6)$ $\fr{so}_{3}^{3}+\fr{u}_{1}^{2, 2}$   &                                             $\fr{so}_5+\fr{u}_1$     \  & no almost effective coset & $\times$ & $\times$       & $\times$  \\  
      &  $\be_7)$ $\fr{so}_{3}^{(3, 3)}+\fr{u}_1^{2, 2, 2}$                             & $\fr{so}_{5}+\fr{u}_1$ &   ${\rm Gr}_2(\mathbb{R}^5) \wi{\times}\Ss^{1}=\displaystyle\frac{\SO_5}{\SO_3\times \SO_2}\wi{\times}\Ss^{1}$ & $\times$ & $\times$   & $\times$     \\
       &  $\be_8)$  $\fr{so}_{3}^{5}+\fr{u}_1^{2}$            &                    $\fr{su}_3+\fr{su}_2$   & $Q_{2}^{7}=\displaystyle\frac{\SU_{3}}{\SO_{3}}\times\frac{\SU_2}{\U_1} \ = \ \displaystyle\frac{\SU_{3}}{\SO_{3}}\times\Ss^{2}$   & $\times$ & $\times$ & 1 $g_{\rm sym}$      \\

           &     $4\fr{u}_{1}$                                & $\fr{g}^{11}\supset 4\fr{u}_{1}$   & no almost effective coset & $\times$ & $\times$       & $\times$  \\  
  \hline
            \end{tabular}
          \]
          \[
\begin{tabular} {r | l | l | l | c | c | c}
$d$  & $\fr{h}$ & $\fr{g}\equiv\fr{g}^{d+7}$ &   $M^{7}=G^{d+7}/H^{d}$ & $\G_2^{\rm inv}$ & $\text{np}\G_2^{\rm inv}$ & $\mc{E}_{\rm inv}$ \\
\thickline

$d>4$ &  $\text{Then}  \ r=2,3$  &    & & &\\
           &   $\text{\bf Case (I)} : \ r=2$ & & & &\\
           \thickline
$d=6$ & $\fr{so}_{4}=\fr{su}_{2}+\fr{su}_{2}'$                 &  $3\fr{su}_2+ 4\fr{u}_{1}$                        & no almost effective coset & $\times$ & $\times$       & $\times$ \\  

           &                                        & $4\fr{su}_2+ \fr{u}_{1}$                           & $\displaystyle\frac{\SU_{2}\times \SU_{2}}{\Delta \SU_{2}}\times \displaystyle\frac{\SU_{2}\times \SU_{2}}{\Delta \SU_{2}}\times \Ss^{1}$  & $\checkmark$ & $\times$ & $\times$        \\ 
           &                                        & $\fr{su}_3+ 5\fr{u}_{1}$                           &  no almost effective coset  & $\times$ & $\times$     & $\times$   \\   
           &                                       & $\fr{so}_{5}+ 3\fr{u}_{1}$                            & $\Ss^{4}\times{\rm T}^{3}= \displaystyle\frac{\SO_{5}}{\SO_{4}}\times{\rm T}^{3} $ & $\times$ & $\times$ & $\times$       \\
           &                                        & $\fr{so}_{5}+\fr{su}_2$                            & $\Ss^{4}\times\Ss^{3}=\displaystyle\frac{\SO_{5}}{\SO_{4}}\times\SU_{2}$ & $\times$ & $\times$  & 1 $g_{\rm sym}$     \\
           &  $\fr{so}_{4}^{(4, 3)}=\fr{su}_{2}+\fr{su}_{2}^{c}$                                      & $\fr{sp}_{2}+\fr{sp}_{1}$                              & $\Ss^{7}_{V^{4}+\bb{R}^{3}}= \displaystyle\frac{\Sp_{2}\times\Sp_{1}}{\Sp_{1}\times\Delta\Sp_{1}}$ & $\checkmark$ & $\checkmark$ & 1      \\
          \hline

$d=8$ &    $\fr{su}_3$                   & $\fr{su}_{4}\supset\fr{su}_{3}$           & $\Ss^{7}_{V^{6}+\bb{R}}=\displaystyle\frac{\SU_{4}}{\SU_{3}}$ & $\checkmark$ & $\checkmark$ & 1 $g_{\rm stn}$ \\
           &                                          & $\fr{g}_{2}+\fr{u}_{1}$                          & $\Ss^{6}_{\rm irr}\times\Ss^{1}=\displaystyle\frac{\G_2}{\SU_{3}}\times{\rm S}^{1}$ & $\checkmark$ & $\times$ & $\times$  \\
\hline
 $d=10$ &  $\fr{so}_{5}$                                       &  $\fr{so}_{6}+2\fr{u}_{1}$                           & $\Ss^{5}_{\rm sym}\times {\rm T}^{2}= \displaystyle\frac{\SO_{6}}{\SO_{5}}\times{\rm T}^{2}$  & $\checkmark$ & $\times$ & $\times$  \\
              \hline
$d=14$ &   $\fr{g}_{2}$                & $\fr{so}_{7}\supset\fr{g}_{2}$                                                     & $\Ss^{7}_{\rm irr}=\displaystyle\frac{\SO_{7}}{\G_2}$ & $\checkmark$ & $\checkmark$ & 1 $g_{\rm irr}$ \\
   \thickline
             &   $\text{\bf Case (II)} : \ r=3$ & & & & \\
             \thickline
               $d=5$ &  $\fr{su}_{2}+2\fr{u}_{1}$  & $\fr{g}^{12}\supset \fr{su}_{2}+2\fr{u}_{1}$ &  no almost effective coset & $\times$ & $\times$ & $\times$ \\
               \hline
                          $d=7$ &  $\fr{so}_{4}+\fr{u}_{1}$ & $\fr{so}_{5}+\fr{su}_{2}+\fr{u}_{1}$ & $\Ss^{1}\wi{\times}\displaystyle\frac{\SU_{2}}{\U_{1}}\times\displaystyle\frac{\SO_{5}}{\SO_{4}}=\Ss^{1}\wi{\times}\Ss^{2}\times\Ss^{4}$ & $\times$ & $\times$ & $\times$  \\
\hline

                          $d=9$ &  $\fr{u}_{3}$ & $\fr{su}_{4}+\fr{u}_{1}$ & $\displaystyle\frac{\SU_{4}}{\U_{3}}\wi{\times}\Ss^{1}=\bb{C}P^{3}\wi{\times}\Ss^{1}$ & $\times$ & $\times$ & $\times$  \\
                                      &   $3\fr{su}_2=\fr{so}_{4}+\fr{su}_{2}$                                    & $\fr{so}_{5}+ \fr{so}_{4}$                            & $\displaystyle\frac{\SO_{5}}{\SO_{4}}\times\frac{\SU_{2}\times \SU_{2}}{\Delta \SU_{2}}= \Ss^{4}\times\Ss^{3} $ & $\times$ & $\times$ & 1 $g_{\rm sym}$ \\

                          \hline
$d=11$ & $\fr{so}_{5}+\fr{u}_{1}$ & $\fr{so}_{6}+\fr{so}_{3}$ & $\displaystyle\frac{\SO_6}{\SO_5}\times  \displaystyle\frac{\SO_3}{\SO_2}   = \Ss^{5}_{\rm sym}\times\Ss^{2}$ & $\times$ & $\times$ & 1 $g_{\rm sym}$ \\
                          \hline

$d=15$ & $\fr{su}_4=\fr{so}_6$                 & $\fr{g}^{22}\supset\fr{su}_{4}$          & no almost effective coset  & $\times$ & $\times$ & $\times$ \\ 
\hline
 $d=28$ & $\fr{so}_{7}$                 & $\fr{so}_{8}\supset\fr{so}_{7}$                                & $\Ss^{7}_{\rm sym}=\displaystyle\frac{\SO_{8}}{\SO_{7}}$   & $\times$ & $\times$ & 1 $g_{\rm sym}$  \\
 \thickline
                       \end{tabular}
\]
\end{center}}

 Table 2 implies   the  following  classification theorem.

\begin{theorem} \label{classg2}
A 7-dimensional compact connected almost effective homogeneous   manifold $M^{7}=G/H$ of a compact Lie group $G$, is diffeomorphic  either to the flat tours ${\rm T}^{7}$ or to a homogeneous manifold of the following list (up to covering)
\[
\begin{tabular}{l|l|l|l}
    $ \ {\rm S}^{7}= \displaystyle\frac{\SO_{8}}{\SO_{7}}= \frac{\SU_{4}}{\SU_{3}}= \frac{\SO_{7}}{\G_2}=\frac{\Sp_{2}}{\Sp_{1}}$ & $ \ \Ss^{3}\times{\rm T}^{4}$ &   $ \    \Ss^{3}\times\bb{C}P^{2}$ &      $ \  \bb{V}_{4, 2}\wi{\times}{\rm T}^{2}$\\ 
    \ \ \ \ \ \ \   $= \displaystyle\frac{\Sp_{2}\times\U_1}{\Sp_{1}\times\Delta\U_{1}}=\displaystyle\frac{\Sp_{2}\times\Sp_{1}}{\Sp_{1}\times\Delta\Sp_{1}}$  &    $ \ \Ss^{4}\times{\rm T}^{3} $ &           $ \ \bb{C}P^{1}\wi{\times}{\rm T}^{5}$ & \    ${\rm Gr}_{2}(\bb{R}^{5})\wi{\times}\Ss^{1}$  \\ 
  $ \ \Ss^{2}\times\Ss^{2}\times\Ss^{2}\wi{\times}\Ss^{1}$          &               $\ \Ss^{5}\times{\rm T}^{2}$                    & $ \ \bb{C}P^{2}\wi{\times}{\rm T}^{3}$      &   $ \ M_{a, b, c}=\displaystyle\frac{\Ss^{3}\times\Ss^{3}\times\Ss^{3}}{\U_1\times\U_1}$  \\
                         $ \ \Ss^{3}\times\Ss^{3}\times\Ss^{1}$    &      $   \  \Ss^{5}\times\Ss^{2}$                                                                                                                    &  $ \ \bb{C}P^{3}\wi{\times}\Ss^{1}$   & $ \ B^{7}= \SO_{5}/\SO_{3}^{\rm ir}$  \\
                         $ \ \Ss^{4}\times\Ss^{2}\wi{\times}\Ss^{1}$      &    $ \ \Ss^3  \times \Ss^4$     &    $ \ \bb{F}_{1, 2}\wi{\times}\Ss^{1}$  &$ \ \bb{V}_{5, 2}\cong T^{1}\Ss^{3}=\SO_{5}/\SO_{3}^{\rm st}$ \\
                         $ \ \Ss^{3}\times\Ss^{2}\times\Ss^{2}$      &  $ \  \Ss^{6}\times\Ss^{1}$         & $  \ W_{k, l}=\displaystyle\frac{\SU_{3}}{\U_1^{k, l}}$        & $ \ N_{a, b}=\displaystyle\frac{\SU_{2}\times \SU_{3}}{\SU_{2}\times\U_{1}}$ \\
                        $ \ \Ss^{3}\times\Ss^{2}\wi{\times}{\rm T}^{2}$                                                                                                                      & $Q^{7}_{1}=\displaystyle\frac{\SU_{3}}{\SO_{3}}\times {\rm T}^{2}$  & $Q^{7}_{2}=\displaystyle\frac{\SU_{3}}{\SO_{3}}\times\Ss^{2}$  & \ $W_{1, 1}=\displaystyle\frac{\SU_{3}\times\SU_{2}}{\SU_{2}^{c}\times\U_1}$ \\
                    \  $\Ss^{2}\times\Ss^{2}\wi{\times} {\rm T}^{3}$
\end{tabular}
\]
Notice that several manifolds in this list admit several presentations as homogeneous spaces, e.g. $\Ss^3, \Ss^5, \Ss^7$,   $\mathbb{C}P^3 $, $\bb{C}P^{3}\wi{\times}\Ss^{1}$, $\Ss^{5}\times\Ss^{2}$, $\bb{V}_{4, 2}\wi{\times}{\rm T}^{2}$,  $\Ss^{3}\times\Ss^{3}\times\Ss^{1}$ and other (for details see Table 2).


\end{theorem}

\subsection{$(4, 7)$-decomposable  homogeneous supergravity backgrounds of  Type III$\al$.}
 The  classification of  compact simply-connected homogeneous weak  $\G_2$-manifolds \cite{FKMS} and that of homogeneous Lorentzian Einstein 4-manifolds \cite{Kom, Fels},  together   with   Theorem \ref{mainresult} yield  a large  list of  $(4, 7)$-decomposable   {\it homogeneous} supergravity backgrounds of  type III$\al$.  Recall   that a $\G_2$-manifold $(M^{7}, \omega)$ is called homogeneous if there is a transitive Lie group $G$ which leaves $\omega$ invariant.
 A classical result of Dynkin states that the Lie algebras  $ \mathfrak{so}_{3}^{7} $,  $\fr{so}_{4}^{(4, 3)}= \mathfrak{su}_{2}+ \mathfrak{su}_{2}^{c}$  and $\mathfrak{su}_{3}$  exhaust (up to conjugation)  all maximal    subalgebras  of  $\mathfrak{g}_2$.  Hence,  a homogeneous manifold  $M^{7} = G/H$ admits an invariant
 $G_2$-structure $\phi$ if and only if $M^{7}=\Spin_{7}/\G_2$ or $\chi_{*}(\mathfrak{h})$ belongs  to one of the subalgebras  $\fr{so}_{3}^{7}$, $\fr{so}_{4}^{(4, 3)}$ and $\fr{su}_{3}$.
Following the papers \cite{Le, Rei} and  \cite{FKMS} in Table 2 we  also indicate which of  the  compact almost effective homogeneous 7-manifolds $M^{7}=G/H$   admit  an invariant $\G_2$-structure and moreover an invariant weak $\G_2$-structure.  To track this information we use the notations ``$\G_2^{\rm inv}$''  and ``$\text{np}\G_2^{\rm inv}$'', respectively. For convenience, in the last column we also  include the number $\mc{E}_{\rm inv}$ of non-isometric invariant Einstein metrics, see also \cite{Cast, Duff2, FKMS, Nik} and Remark \ref{einrem} below.  By ``$\times$'' we   mean that the corresponding coset does not admit some of the aforementioned  invariant objects.

\begin{remark}\label{einrem}\textnormal{({\bf Remarks on Table 2 about homogeneous Einstein metrics})
For the  homogeneous spheres $\Ss^{5}, \Ss^{6}$ and $\Ss^{7}$ in Table 2 we  use a subscript with  the decomposition of the associated  tangent space into irreducible submodules, in particular the subscript  ``${\rm irr}$'' characterises an irreducible isotropy representation   (but not symmetric), while  ``${\rm sym}$''  means that  the corresponding sphere is  a symmetric space (and similarly for the metrics).  The  space $M_{a, b, c}$ is diffeomorphic to  $\Ss^{2}\times\Ss^{2}\times\Ss^{3}$ and is a  circle bundles over $\Ss^{2}\times\Ss^{2}\times\Ss^{2}$.  Details about the number of invariant Einstein metrics on $M_{a, b, c}$, which depends on the parameters $(a, b, c)$, can be found in \cite{Nik},  for example.   
 The Berger sphere $B^{7}$ and the 7-spheres  $\Spin_7/\G_2$ or $(\Sp_2\times\Sp_1)/(\Sp_1\times\Delta\Sp_1)$ admit a unique invariant proper weak  $\G_2$-structure, see \cite{Br1, Bar, FKMS} and   a unique invariant Einstein metric. In fact,  this structure on the squashed sphere $(\Sp_2\times\Sp_1)/(\Sp_1\times\Delta\Sp_1)$ is also invariant under the Lie group $\Sp_{2}\times\U_1$. Recall  now that the Allof-Wallach spaces $W_{k, l}=\SU_{3}/\U_{1}^{k, l}$, where $\U_{1}^{k, l}={\rm diag}(z^{l}, z^{k}, \bar{z}^{l+k})\subset \U_2\subset\SU_3$ with $z\in \Ss^{1}=Z(\U_2)$, $k\geq 1$, $l\geq 1$, ${\rm gcd}(k, l)=1$,    admit (up to homothety)  two  $\SU_3$-invariant weak   $\G_2$-structures and  two invariant Einstein metrics, see \cite{FKMS, Nik}.   These Einstein metrics are isometric each other for the special case of $W_{1, 0}$, in particular the weak $\G_2$-structures on $W_{1, 0}$ coincide.   By \cite{Boy} it is  also known that  the exceptional Allof-Wallach space $W_{1, 1}=(\SU_3\times\SU_2)/(\SU_{2}^{c}\times\U_1)$ and  the 7-sphere $\Ss^{7}=\Sp_2/\Sp_1$ exhaust all compact homogenous 3-Sasakian spaces in dimension seven. Note that a 7-dimensional 3-Sasakian manifold admits a second weak $\G_2$-structure which is proper, with the corresponding Einstein metric  to be a member of the canonical variation of the invariant 3-Sasakian Einstein metric, see \cite{FKMS}.  Recall also that the Stiefel manifold  $\bb{V}_{5, 2}$ is an Einstein-Sasakian manifold and  the unique $\SU_4$-invariant Einstein metric on $\SU_4/\SU_3$ is the standard one, $g_{\rm stn}$, see  \cite{Jen}. Finally notice that the homogeneous spaces  $Q^{7}_{1}=(\SU_{3}/\SO_{3})\times {\rm T}^{2}$  and  $Q^{7}_{2}=(\SU_{3}/\SO_{3})\times\Ss^{2}$
are products of the symmetric space $\SU_{3}/\SO_{3}$ with the 2-torus ${\rm T}^{2}$ and the 2-sphere $\Ss^{2}$, respectively.  The coset $\SU_3/\SO_3$  belongs to the family $\SU_{n}/\SO_{n}$, which according to \cite{Cahen} is spin only for $n=\text{even}$. Consequently, none of $Q_{1}^{7}$ and $Q_{2}^{7}$ are  spin or admit a $\G_2$-structure (see  Proposition \ref{topology}).  A difference between the symmetric spaces $Q_1^{7}, Q_{2}^{7}$  is that $Q_{1}^{7}$  is not simply-connected neither Einstein, in contrast to $Q_{2}^{7}$ which satisfies both these properties (it admits a unique invariant Einstein metric given by the product of the Killing metrics). }
\end{remark}



\subsection{Non existence  of invariant $\G_2$-structures and invariant $\G_2^{*}$-structures}
Let us describe now all compact almost effective homogeneous spaces $M^{7}=G/H$ which admit no $G$-invariant $\G_2$-structure and moreover no $\G_2$-structure.
This task  is based  on our classification Theorem \ref{classg2}, the column ``$\G_2^{\rm inv}$'' of Table 2 and Proposition \ref{topology}. We conclude the following

\begin{theorem}\label{NONGEN}
1) \  Let    $M^7 = G/H$ be a  compact connected   almost  effective  homogeneous 7-manifold   of  a   compact Lie  group  $G$.
  The  manifold  $M^{7}$   admits no $G$-invariant $\G_2$-structure (or equivalently, no $G$-invariant spin structure)  if and only if it is diffeomorphic (up to covering) to one of the following cosets:
\[
  \begin{tabular}{l | l }
  $\text{spin}$ & $\text{non-spin}$ \\
  \thickline
 $\Ss^{3}\times\Ss^{4}=(\SU_{2}\times \SU_{2}/\Delta \SU_{2})\times(\SO_5/\SO_4)$ & $\bb{C}P^{2}\times\Ss^{3}=(\SU_3/\U_2)\times\SU_2$   \\
 $\Ss^{4}\times {\rm T}^{3}=(\SO_5/\SO_4)\times {\rm T}^{3}$ & $\bb{C}P^{2}\wi{\times} {\rm T}^{3}=(\SU_3/\U_2)\wi{\times}{\rm T}^{3}$ \\
 $\Ss^{2}\times\Ss^{2}\times\Ss^{2}\times\Ss^{1}=(\SU_2/\U_1)^{3}\times\Ss^{1}$         &   $Q_{1}^{7}=(\SU_{3}/\SO_{3})\times {\rm T}^{2}$   \\
  $\Ss^{2}\times\Ss^{5}=(\SO_3/\SO_2)\times(\SO_6/\SO_5)$ & $Q_{2}^{7}=(\SU_{3}/\SO_{3})\times\Ss^{2}$   \\
$\bb{C}P^{1}\wi{\times}  {\rm T}^{5}=(\SU_2/\U_1)\wi{\times}{\rm T}^{5}$  &${\rm Gr}_{2}(\bb{R}^{5})\wi{\times}\Ss^{1}$  \\
 $\Ss^{2}\times\Ss^{2}\widetilde{\times} {\rm T}^{3}=(\SU_2\times\SU_2/\U_1\times\U_1)\wi{\times}{\rm T}^{3}$  &\\
 $\Ss^{3}\times\Ss^{2}\wi{\times}{\rm T}^{2}=(\SU_{2}\times \SU_{2}/\Delta \SU_{2}) \times(\SU_{2}/\U_1)\wi{\times}  {\rm T}^{2}$ & \\
 $\Ss^{4}\times\Ss^{2}\wi{\times}\Ss^{1}=(\SO_5/\SO_4)\times(\SO_3/\SO_2)\wi{\times}\Ss^{1}$ & \\
 $\bb{C}P^{3}\times\Ss^{1}=(\SU_4/\U_3)\wi{\times}\Ss^{1}$ & \\
 $\Ss^{7}=\SO_8/\SO_7$
                  \end{tabular}
                   \]
  2) \  Manifolds  from  the left  column  admit  a $\G_2$-structure which is not invariant, or in other words, admit  a generic  3-form which is not  invariant.
Inside the class of compact connected almost effective homogeneous   7-manifolds $M^{7}=G/H$ only the manifolds  from  the  right  column  doest not  admit  a $G_2$-structure.              \end{theorem}

\noindent  Theorem \ref{NONGEN}  gives rise to the following natural  questions for further research.

\medskip
{\bf Question 1.}  What is the explicit form of the non-invariant spin structure, or equivalent, non-invariant $\G_2$-structure assigned in Theorem \ref{NONGEN}? 

\medskip
{\bf Question 2.} What is the  symmetry group corresponding to such a structure?

\medskip
\noindent  These type of questions are in general difficult. To our knowledge,  they have been  examined for example  in \cite{Le2} for the coset $\Ss^{3}\times\Ss^{4}$ and for $\G_2^{*}$-structures.  Below we also describe our conclusions for non-existence of $\G_2^{*}$-structures. But firstly, let us   analyse some example  and enlighten the details of Theorem \ref{NONGEN}.
\begin{example}
\textnormal{The space $\Ss^{3}\times\Ss^{4}$ is  a spin manifold and  by Proposition \ref{topology}, also a $\G_2$-manifold.  However, this $\G_2$-structure   is not  invariant  with respect to  $G=\SO_{5}\times\SU_{2}$, where we identify  $\Ss^{3}\times\Ss^{4}\cong\SU_{2}\times(\SO_{5}/\SO_{4})$. Indeed,  a spin structure on a  seven-dimensional oriented  connected homogeneous Riemannian manifold $(M^{7}=G/H, g)$ with a reductive decomposition $\fr{g}=\fr{h}+\fr{m}$ is {\it invariant} if the isotropy representation $\chi : H\to\SO(\fr{m})$ lifts to $\Spin(\fr{m})\cong\Spin_{7}$, i.e. there exists a homomorphism $\hat{\chi} : H\to\Spin(\fr{m})$ which makes the following diagram commutative
  \[
    \xymatrix{
                                                           &   \Spin_{7} \ar[d]^{\Ad}    \\
                    K \  \ar[r]^{\chi}  \ar[ur]^{\hat{\chi}} &   \SO_{7}.    }
\]
Here, $\Ad : \Spin_{7}\to\SO_{7}$ is the double covering. Conversely, if $G$ is simply-connected and $(M^{7} = G/H, g)$ has a spin structure, then $\chi$ lifts to $\Spin(\fr{m})$,  i.e. the spin structure is $G$-invariant (see \cite[Thm.1, p.~146]{Cahen2}).  Hence in this case there is a bijective correspondence between the set of spin structures on $(M^{7} = G/H, g)$  and the set of lifts of $\chi$ onto $\Spin(\fr{m})$. If in addition $M=G/K$ is simply-connected and such a lift exists, then it will be unique.  For the product  $\Ss^{3}\times\Ss^{4}=\SU_{2}\times(\SO_{5}/\SO_{4})$ the full isometry group $G=\SO_{5}\times\SU_{2}$ is not simply-connected, so the spin structure which admits $\Ss^{3}\times\Ss^{4}$ does not lift to a $G$-invariant spin structure, or in other words the corresponding $\G_2$-structure is not $G$-invariant.    
All the spaces  in Theorem \ref{NONGEN} which are spin can be justified in a similar way.}
  \end{example}

\noindent {\bf Results about $\G_2^{*}$-structures.} Recall  that in a line  with  a $\G_2$-structure, a compact manifold $M^{7}$ admits a $\G_2^{*}$-structure if and only if $M^{7}$ is orientable and spin, see \cite[Main Theorem]{Le3}.  On the other hand, recall that $\SO_4$ is  the unique maximal compact subgroup of $\G_2^{*}$, but also a maximal subgroup $\G_2$. Therefore, in the homogeneous setting we see that a $G$-invariant $\G_2^{*}$-structure on a compact homogeneous space $M^{7}=G/H$ induces also a $G$-invariant $\G_2$ structure. However, the converse does not always   true, since given a compact connected coset $M^{7}=G/H$ such that $\chi(H)\subset\G_2$, then we may have $\chi(H)\nsubseteq \G_2^{*}$.  In fact, this is the case for the invariant $\G_2$-structures on the cosets
\begin{equation}\label{nog22}
 B^{7}=\displaystyle\frac{\SO_{5}}{\SO_{3}^{\rm ir}}, \quad \frac{\Spin_{7}}{\G_2},\quad \frac{\SU_4}{\SU_3}, \quad \frac{\G_2}{\SU_3}\times\Ss^{1}.
\end{equation}
In \cite{Le2} one obtains  the non-existence of   invariant $\G_2^{*}$-structures on the product $\Ss^{3}\times\Ss^{4}$.  Next we classify all   compact almost effective homogeneous spaces $M^{7}=G/H$ which can be characterised by the same non-existence.  
\begin{corol}
 1) \ A seven-dimensional compact connected almost effective  homogenous manifold $(M^{7}=G/H, g)$ of a  connected compact Lie group $G$ which  admits no $G$-invariant $\G_2^{*}$-structure is diffeomorphic (up to covering) to one of the cosets given in Theorem \ref{NONGEN}, 1), or one of the cosets given in (\ref{nog22}). \\
 2) \ Inside the class of compact connected almost effective homogeneous 7-manifolds $M^{7}=G/H$ only the manifolds $\bb{C}P^{2}\times\Ss^{3}$, $\bb{C}P^{2}\wi{\times} {\rm T}^{3}$, ${\rm Gr}_{2}(\bb{R}^{5})\wi{\times}\Ss^{1}$ and $Q_{1}^{7}, Q_{2}^{7}$ do not admit a $\G_2^{*}$-structure. 
\end{corol}

 \section{Some   solutions   of  the Maxwell  equation   for  non  generic   3-forms}

Next we present  examples of  compact  homogeneous Riemannian manifolds  $(M^{7}=G/H, g)$ which admit  {\it non-generic}  invariant  special  3-forms, that means   3-forms $\phi$ which  satisfy  the Maxwell equation  $\dd \phi =  f\star_{7}\phi$ and are of type III$\be$. 

\subsection{Solution  of  Type III$\be$ for the Maxwell   equation    on  $M^7 = \bb{CP}^2 \times \Ss^3$.}

The  simply-connected   homogeneous manifold
$M^7 = \bb{CP}^2 \times \Ss^3   =  (\SU_{3}/\U_{2})  \times \SU_{2} $ has no  spin  structure.  Hence there are not exist generic 3-forms.
However, here  we  will  show   that  it  is endowed with  invariant  (non-generic) special 3-forms.

  The Lie algebra $\fr{g}=\fr{su}_{3} + \fr{su}_{2}$ admits  the reductive decomposition
\[
\fr{g}  = \fr{h}+ \fr{m},\quad   \fr{h}  = \fr{u}_2, \quad  \fr{m}  = \fr{m}_1 + \fr{m}_2 = \bb{R}^4 + \fr{su}_{2}.
\]
The tangent space at the identity in $M^{7}$ can be identified with $\fr{m}$. Dually, we have $\fr{g}^* = \fr{m}_1^* + \fr{m}_2^* + \fr{h}^*$ where we identify $\fr{m}^* = \fr{m}_1^*+\fr{m}_2^*$ with the cotangent space at the identity. One can choose a basis adapted to this decomposition of $\fr{g}^*$: $\fr{m}_1^* = \mathrm{span}( \alpha^i )_{i=1,\ldots ,4} $, $\fr{m}_2^* =\{ \beta^i \}_{i=1 , \ldots, 3 }$, $\fr{h}^* =\{ \gamma^i \}_{i=1, \ldots , 4}$. Note that $\mathrm{Ann}( \fr{m}_1 ) = \fr{m}_2^* + \fr{h}^*$, $\mathrm{Ann}( \fr{m}_2 ) = \fr{m}_1^* + \fr{h}^*$ and $\mathrm{Ann}( \fr{h} ) = \fr{m}_1^* + \fr{m}_2^*$. The structure equations then read
\begin{align*}
\dd \alpha^1 & = - \alpha^2 \wedge \gamma^3 - \alpha^3 \wedge \left( 3 \, \gamma^1 - \gamma^2 \right) - \alpha^4 \wedge \gamma^4 \, ,\quad \dd \gamma^1  = - \alpha^1 \wedge \alpha^3 - \alpha^2 \wedge \alpha^4 \, ,\\
\dd \alpha^2 & = \alpha^1 \wedge \gamma^3 - \alpha^3 \wedge \gamma^4 - \alpha^1 \wedge \left( 3 \, \gamma^1 + \gamma^2 \right) \, ,\quad  \dd \gamma^2  = \alpha^1 \wedge \alpha^3 - \alpha^2 \wedge \alpha^4 - 2 \, \gamma^3 \wedge \gamma^4 \, , \\
\dd \alpha^3 & = \alpha^1 \wedge \left( 3 \, \gamma^1 - \gamma^2 \right) + \alpha^2 \wedge \gamma^4 - \alpha^4 \wedge \gamma^2 \, , \quad \dd \gamma^3  = - \alpha^1 \wedge \alpha^2 - \alpha^3 \wedge \alpha^4 - 2 \, \gamma^4 \wedge \gamma^2 \, , \\
\dd \alpha^4 & = \alpha^1 \wedge \gamma^4 + \alpha^2 \wedge \left( 3 \, \gamma^1 + \gamma^2 \right) - \alpha^3 \wedge \gamma^3 \, , \quad \dd \gamma^4  = - \alpha^1 \wedge \alpha^4 - \alpha^2 \wedge \alpha^3 - 2 \, \gamma^2 \wedge \gamma^3 \, , \\
\dd \beta^1 & = - \beta^2 \wedge \beta^3 \, ,\quad \dd \beta^2  = - \beta^3 \wedge \beta^1 \, ,  \quad \dd \beta^3  = - \beta^1 \wedge \beta^2 \, .
\end{align*}
Any $\U_2$-invariant metric on $M^{7}$ has  the form
$  g  = g_4 + g_3 $
where   $g_4  = a \, \sum_{i=1}^4 \alpha^i \otimes \alpha^i $  is proportional  to  the Fubin-Strudy metric  and  $g_3$  is   any Euclidean metric on
$\mathfrak{su}_3$. Without loss of generality,  we  may  assume  that
$g_3  = \sum_{i=1}^3 c_i \, \beta^i \otimes \beta^i$,
for some positive constants $c_i$ (see \cite{Milnor}). 	Denote by $\vol_4 = a^2\cdot(\alpha^1 \wedge \alpha^2 \wedge \alpha^3 \wedge \alpha^4)$ the volume form induced from $g_4$ on $\bb{CP}^2$ and by $\vol_3 = \sqrt{c_1 c_2 c_3}\cdot(\beta^1 \wedge \beta^2 \wedge \beta^3)$ the volume form on $\Ss^3$ induced from $g_3$. Then, the metric-compatible volume form is given by $\vol_7 = \vol_4 \wedge \vol_3$.

Now, the most general $\U_2$-invariant $3$-form on $M^{7}$ is  given  by
\begin{equation}\label{3form1}
\phi  = \omega \wedge \theta + b \cdot \vol_3,
\end{equation}
where $\omega = a\cdot \left( \alpha^1 \wedge \alpha^3 + \alpha^2 \wedge \alpha^4 \right)$ is the K\"{a}hler form on $\bb{CP}^2$, $\theta$ is an arbitrary $\SU_{2}$-invariant $1$-form on $\Ss^3$ and $b$ a constant. It is straightforward to check that $\omega$ is anti-self-dual, i.e.\ $\star_4 \omega = - \omega$. In particular, we have $\star_7 \phi  = - \omega \wedge \star_3 \theta + b \cdot \vol_4.
$
Computing the exterior derivatives, we find
\[
\dd \star_7 \phi  = - \omega \wedge \dd \star_3 \theta, \quad   \dd \phi  = \omega \wedge \dd \theta \, .
\]
From the structure equations we also see  that any $2$-form on $\SU_{2}$ is closed and thus $\theta$ must be co-closed, i.e.\ $\dd \star_3 \theta = 0$. Hence, the equation $\dd \star_7 \phi  = 0$  
 is always satisfied. Now, the Maxwell equation $\dd \phi  =f \star_7 \phi$ 
 reads as
\begin{align*}
\omega \wedge \dd \theta & =f\cdot \left( - \omega \wedge \star_3 \theta + b \cdot \vol_4 \right).
\end{align*}
Matching each side of the equation yields the following conditions:
\[
\dd \theta  = - f \star_3 \theta, \quad  f\cdot b \cdot \vol_4  = 0 \, .
\]
Taking the components of the first of these equations leads to
\begin{equation}\label{syt1}
\left( - \sqrt{\frac{c_1}{c_2 c_3}}  + f \right) \theta_1  = 0, \quad
\left( - \sqrt{\frac{c_2}{c_3 c_1}}  + f \right) \theta_2  = 0, \quad
\left( - \sqrt{\frac{c_3}{c_1 c_2}}  +f \right) \theta_3  = 0.
\end{equation}
Thus, there are two non-trivial cases to examine:
\begin{itemize}
\item
If $f = 0$, then we automatically get $\dd \theta = 0$, which implies $\theta=0$ by the last system of equations. Thus, (\ref{3form1}) reduces to $\phi = b\cdot \vol_3$.
\item If $f \neq 0$, then we obtain $b=0$ so that (\ref{3form1}) reduces to  $\phi  = \omega \wedge \theta \, .$
\end{itemize}
\begin{prop}
The  only  invariant solutions  of  the  Maxwell  equation  on $M^7 = \bb{CP}^2 \times \Ss^3$ are the following:
\begin{itemize}
\item if $f = 0$, $\phi = b\cdot \vol_3,\ b =  const$,
\item if $f\neq 0$, $\phi = \omega \wedge \theta$ where  $\omega$  is  the  K\"ahler  form  of  $\mathbb{C}P^2$   and  the components of the $1$-form $\theta$ and of the metric are subject to (\ref{syt1}).
\end{itemize}
\end{prop}

In both cases, one can check that these special 3-forms do not satisfy the supergravity Einstein equation with respect to the metric $g$, hence $M^{7}$  does not provide us with a special gravitational 7-manifold.
\subsection{Solution  of of Type III$\be$ for the  Maxwell   equation   on the Lie group  $ G = \Ss^{3}\times {\rm T}^4$.}
   We   choose  a left invariant metric $g$   on  $G$  such  that  the   decomposition    $\mathfrak{g} = \mathfrak{su}_2 + \mathfrak{t}$   is orthogonal, where we indentify the tangent space of  $\Ss^{3}=\SU_2$ with the Lie algebra $\fr{su}_{2}$ and similarly for the 4-torus ${\rm T}^{4}$, i.e. $\fr{t}=T_{e}{\rm T}^{4}$.
    Then we may choose    and  orthogonal basis   $\omega_{\alpha}$  of  1-forms  on  $\mathfrak{su}_2$    such  that
 $d \omega^{\alpha} = \omega^{\beta} \wedge \omega^{\gamma}$,
where  $(\alpha, \beta, \gamma)$ is  a  cyclic  permutation  of  $(1,2,3)$, and moreover an orthonormal basis $\rho_i, \, i=1,2,3,4$ of  $\mathfrak{t}$ such that $d \rho_i =0$.
Set
\[
\bigwedge^{p,q} = \bigwedge^p(\mathfrak{su}_2^*) \wedge \bigwedge^q (\mathfrak{t}^*).
\]
Then  $\dd \bigwedge^{p,q} \subset \bigwedge^{p+1,q} $  and   $\star_{7} \bigwedge^{p,q} \subset \bigwedge^{3-p,  4 - q}$.
This  show  that   any  solution  of Maxwell equation belongs   to
\[
\bigwedge^{1, 2}  = \mathfrak{su}_2^* \wedge  \bigwedge^2(\mathfrak{t}^*).
\]
   Now, the    space   $\bigwedge^2(\mathfrak{t}^*)= \bigwedge^{+}+\bigwedge^{-}$   is  the  direct  sum  of    self-dual  forms  $\bigwedge^{+}$  and    anti-self-dual  forms   $\bigwedge^{-}$,  which  are the  $\pm$  eigenspaces of  the Hodge operator  $\star_{4}$.
Set $\phi = \omega \wedge \sigma \in  \bigwedge^{1,2}$,
where $\omega$ is a left-invariant 1-form on $\SU_{2}$ and $\sigma\in \bigwedge^2(\mathfrak{t}^*)$ is a left-invariant 2-form on the torus ${\rm T}^{4}$.
Then we get
\[
\dd \phi = d \omega  \wedge \sigma, \quad  \star_{7}\phi = \star_3 \omega \wedge \star_{4}\sigma.
\]
Now, we may  assume  that   $g(\omega^{\alpha}, \omega^{\beta} ) = (\lambda^{\alpha})^{-2} \delta^{\alpha, \beta}$.
In this case it is easy to see that $\tilde{\omega}^{\alpha} = \lambda^{\alpha}\omega $ is  an orthonormal basis   and moreover
\[
\star_3 \omega^{\alpha}  =  \frac{\lambda^{\beta} \lambda^{\gamma}}{\lambda^{\alpha}}\omega^{\beta} \wedge \omega^{\gamma}.
\]
Therefore, $\phi = \omega^{\alpha} \wedge \sigma$  satisfies  the Maxwell  equation  if  and only if
\[
\star_{4} \sigma = \pm \sigma, \quad\text{and}\quad   \lambda^{\beta} \lambda^{\gamma}  = \pm \lambda^{\alpha}.
\]
This  implies that   $ \lambda^{\alpha}  = \pm  1$. More precisely, $(\lambda^1, \lambda^2, \lambda^3)   = (\pm  1, \pm 1, \pm 1)$.
Note that if   $\sigma$ is    self-dual the number of  units in  this   triple must be odd and  if
$\sigma$ is an anti-self-dual the corresponding number is even.
For  example,  assume  that  $\lambda^{\alpha} =1, \,  \alpha  = 1,2,3 $. Then, any  self-dual  2 form  $\sigma \in  \bigwedge^+$   defines  a  solution of Type III$\be$ for the Maxwell, given by  $\phi = \omega \wedge \sigma$,
 where  $\omega$
 is  any  unit  1-form in  $\mathfrak{su}_2^*$.


 \end{document}